\documentclass{amsart}
\usepackage{amssymb}
\usepackage{hyperref}
\usepackage{mathrsfs}
\usepackage{mathtools}
\usepackage{centernot}
\usepackage{bbm}
\usepackage{dirtytalk}

\usepackage[utf8]{inputenc}
\newcommand{\Z}{\mathbb{Z}}	
\newcommand{\R}{\mathbb{R}}	
\newcommand{\N}{\mathbb{N}}

\DeclarePairedDelimiter{\ceil}{\lceil}{\rceil}

\newcommand{\paren}[1]{\left( #1 \right)}
\newcommand{\brac}[1]{\left[ #1 \right]}

\newcommand{\abs}[1]{\left\vert#1\right\vert}

\newcommand{\set}[1]{\left\{#1\right\}}

\newcommand{\interior}[1]{%
  {\kern0pt#1}^{\mathrm{o}}%
}

\DeclareMathOperator{\rank}{rank}
\DeclareMathOperator{\im}{im}

\newtheorem{theorem}{Theorem}
\newtheorem{lemma}[theorem]{Lemma}
\newtheorem{prop}[theorem]{Proposition}

\newtheorem{corollary}[theorem]{Corollary}

\usepackage{tikz}
\usepackage{tikz-cd}

\title{Disks, Surfaces, and Entanglement Percolation}
\author{Paul Duncan}
\author{Benjamin Schweinhart}
\author{David Sivakoff}
\date{\today}

\begin{document}

\begin{abstract}
We study the probability that a loop is null-homotopic --- that is, bounded by the continuous image of a disk --- in plaquette percolation on $\Z^3.$ Locally, the event that there is a ``horizontal disk crossing'' of a rectangular prism is dual to the event that there is a vertical crossing in entanglement percolation (with wired boundary conditions). However, the analysis of analogous events on the full lattice is complicated by the long-range nature of entanglement percolation. We show that the probability that a rectangular loop is contractible exhibits a phase transition from area law to perimeter law dual to the entanglement percolation threshold, conditional on a conjecture concerning the continuity of entanglement percolation thresholds with respect to truncation. We also show the continuity of a truncated entanglement percolation threshold in slabs and apply that to identify a regime where large plaquette surfaces exist but typically have many handles.
\end{abstract}
\maketitle

\section{Introduction}

The study of random graphs has been an active and fruitful area of research for more than 50 years, though much less is known about analogous questions in higher dimensions. Random simplicial and cubical complexes exhibit interesting topological phenomena that are not present in graphs. In the mean-field setting, distinct phase transitions have been exhibited related to the core topological notions of homology~\cite{linial2006homological,meshulam2009homological,linial2016phase} and homotopy~\cite{babson2011fundamental}. For another mean-field model the thresholds for the vanishing of the one-dimensional homology and the fundamental group coincide~\cite{kahle2021topology}. However, the topology of a percolation complex on the lattice is less well understood. 

We are interested in the appearance of giant surfaces in higher dimensional lattice models, in analogy with the giant component from classical percolation. However, different classes of surfaces are expected to appear at different thresholds. In this article we study the resulting topological regimes from two different perspectives. A seminal result of Aizenman, Chayes, Chayes, Fr\"{o}lich, and Russo~\cite{ACCFR83} showed that there is a phase transition in the probability that a large loop is the boundary of a set of plaquettes, a homological notion of ``filling'' the loop. The same authors conjectured that there should be a distinct threshold for the related homotopic event that a large loop is contractible, that is, the event that the loop can be continuously deformed to a point within the plaquette system. As we will see, contractibility turns out to be equivalent to filling the loop with a disk in a thickening of the plaquettes. We partially confirm this conjecture and then investigate the intermediate regime between homological filling and contractibility, showing that the large surfaces that appear have complicated topological structure.

The main tool that we use is entanglement percolation, which we describe in greater detail below. Entanglement is known to percolate at a different threshold than connectivity~\cite{AG91,holroyd2000existence,holroyd2002inequalities}, and is related to the presence of spheres in the dual complex~\cite{grimmett2010plaquettes,grimmett2000entanglement}. However, entanglement percolation is poorly understood compared to classical percolation due to the nonlocality of deciding whether a set of bonds is entangled. We obtain our results by studying an intermediate local form of entanglement.

\begin{figure}
	\includegraphics[height=0.2\textwidth]{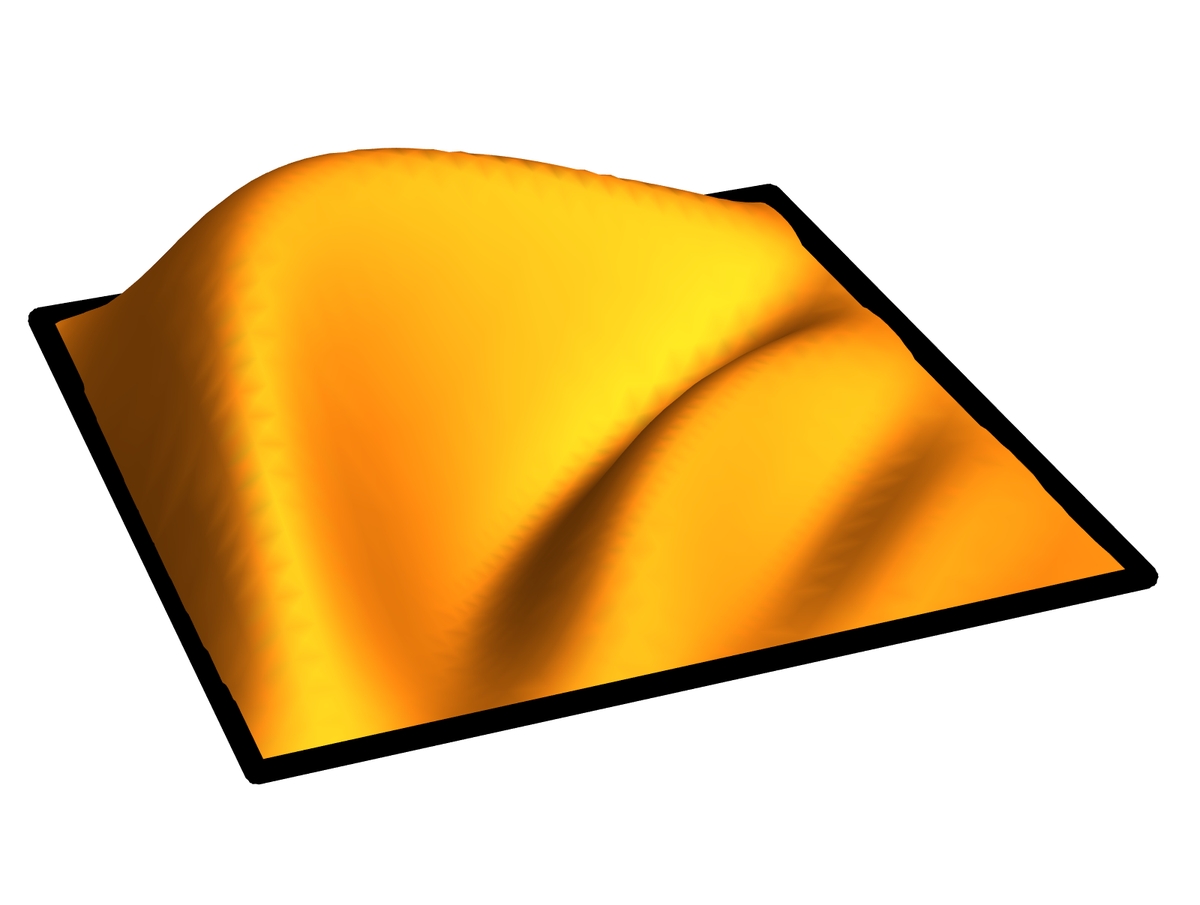}\qquad
	\includegraphics[height=0.2\textwidth]{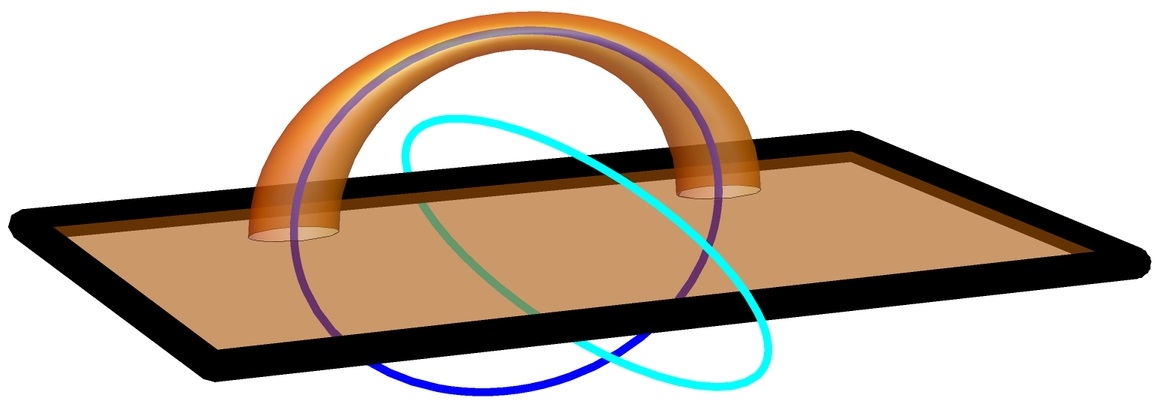}
    \caption{Plaquette events: (left) the loop $\gamma$ (shown in black) is null-homotopic as it is bounded by an orange surface of plaquettes homeomorphic to a disk. (right) The loop $\gamma$ is null-homologous (bounded by a surface) but not null-homotopic. In fact, a null-homotopy is precluded by the existence of two loops of dual bonds (shown in blue) which, together with $\gamma$, form a set of Borromean rings.}
\label{fig:giantcycles}
\end{figure}

\section{Main Results}

Our model of interest is plaquette percolation on the cubical nearest-neighbor lattice with vertices in $\Z^3.$ A (2-dimensional) plaquette is an integer translate of one of the unit squares $\brac{0,1} \times \brac{0,1} \times \set{0},$ $\brac{0,1} \times \set{0} \times \brac{0,1},$ or $\set{0} \times \brac{0,1} \times \brac{0,1}.$ We can also think of the plaquettes as the $2$-dimensional cells of the cubical complex formed by writing $\R^3$ as a union of unit cubes with integer corners. The plaquette system $P = P(p)$ is the random subcomplex obtained by taking the full $\Z^3$ lattice and adding each plaquette independently with probability $p.$ 

For a loop $\gamma$ of edges in $\Z^3,$ we define $U_{\gamma}$ to be the event that $\gamma$ is contractible in $P.$ An equivalent definition is that $U_{\gamma}$ is the event that $0 = \brac{\gamma} \in \pi_1\paren{P},$ where we recall that the fundamental group $\pi_1\paren{P}$ is the group of loops up to homotopy, or continuous deformation. The formal definition of the fundamental group involves a choice of basepoint for the loop, but in a path connected space, the choice turns out not to matter (see~\cite{hatcher2002algebraic} for more information). This event is difficult to study directly, so we approach it using the dual lattice $\paren{\Z^3}^* \coloneqq \Z^3 + \paren{1/2,1/2,1/2}.$ Each bond of $\paren{\Z^3}^*$ intersects exactly one plaquette of $\Z^3,$ so taking $q=1-p,$ we may couple $P$ and a dual bond percolation $B$ with parameter $q$ so that each plaquette appears in $P$ if and only if the intersecting dual bond does not appear in $B.$ 

A set of dual bonds $E$ is called entangled if there is no embedding $\mathcal{S}$ of the 2-sphere $\mathbb{S}^2$ into $\R^3 \setminus E$ so that $E$ intersects both connected components of $\R^3 \setminus \mathcal{S}.$ This definition is more closely connected to the plaquette system than it may appear because there is a continuous map (more specifically, a deformation retraction) from $\R^3 \setminus B$ onto $P$~\cite{duncan2020homological}. As a result, we may consider continuous images of spheres in $P$ instead of embedded spheres in $\R^3 \setminus B,$ and our first step on the way to analyzing the behavior of loops in the plaquette system is to characterize $U_{\gamma}$ in terms of dual entanglement.

\begin{prop}\label{prop:vgammadual}
    Let $\gamma$ be a loop in $\Z^3.$ Then $U_{\gamma}$ occurs if and only if there is no set of dual bonds $E \subset B$ so that $E \cup \gamma$ is entangled.
\end{prop}

We also show that the event of an entangled crossing between opposite sides of a box is dual to the event that the loop generated by the remaining sides is contractible in Proposition~\ref{prop:duality}. These results are in the same vein as the duality between entangled paths and enclosing spheres studied by Grimmett and Holroyd in~\cite{grimmett2010plaquettes}.

Entanglement percolation is difficult to work with because of long-range dependence, so we also introduce a restricted version --- called $m$-entanglement --- in which two vertices of $\paren{\Z^3}^*$ are connected by a long-range edge if they are contained in an entangled set with at most $m$ bonds. Let $C^m(a)$ be the $m$-entangled component of $a\in \paren{\Z^3}^*$, or in other words the connected component of $a$ once the appropriate long-range edges are included. Then we can define the threshold
\[p^{m}_c = \inf\{p\, : \, \mathbb{P}_p\paren{|C^{m}(a)| = \infty} > 0\}\,.\]

We prove that there is an area law for the event $U_{\gamma}$ when the dual bond parameter is above the limit of the former thresholds
\[p^{\lim}_c = \lim_{m \to \infty} p_c^m\,\]
and a perimeter law below the related threshold 
\[\pi_c^e = \inf\{p:\mathbb{E}(|C^e(a)|) = \infty\}\,\]
for the expected size of $C^{e}(a)$, the entangled component of a vertex $a\in \paren{\Z^3}^*.$ 

\begin{theorem}
Let $\gamma$ be a rectangular loop in a lattice plane. Let $U_\gamma$ be the event that $\gamma$ is contractible in $P.$ 

Then there is a constant $\alpha(p)\in (0,\infty)$ such that
  \begin{align*}
      \log \mathbb{P}_p(U_\gamma) 
     \begin{cases}
     \sim  -\alpha(p) \mathrm{Area}(\gamma) &\qquad p < 1-p_c^{\lim}\\
     = - \Theta(\mathrm{Per}(\gamma)) &\qquad p > 1-\pi_c^e
      \end{cases},
  \end{align*}
  as the dimensions of $\gamma$ tend to $\infty$.
\end{theorem}

We note that the $p^{m}_c$ are strictly decreasing in $m$ by the results of~\cite{AG91,balister2014essential}. In particular, $p_c^{\lim}<p_c\paren{\Z^3}$ so this transition is distinct from that for the probability that a loop is nullhomologous as considered in~\cite{ACCFR83}.
To show that this transition is sharp, it would be enough to show that 
\[p_c^{\lim} = \pi_c^e\,,\]
which remains an open question. It follows from the techniques of~\cite{duminil2019sharp} that $p_c^m = \pi_c^m$ for each $m,$ where $\pi_c^m$ is the threshold for infinite expected cluster sizes in $m$-entanglement percolation. The sharpness of the transition can then be thought of in terms of continuity of these two thresholds in $m.$ For a different class of enhancements in $\Z^2$ we show that $p_c^{\lim}=p_c^e$ in a companion paper~\cite{duncan2025enhancement}.

\begin{figure}
	\includegraphics[height=0.4\textwidth]{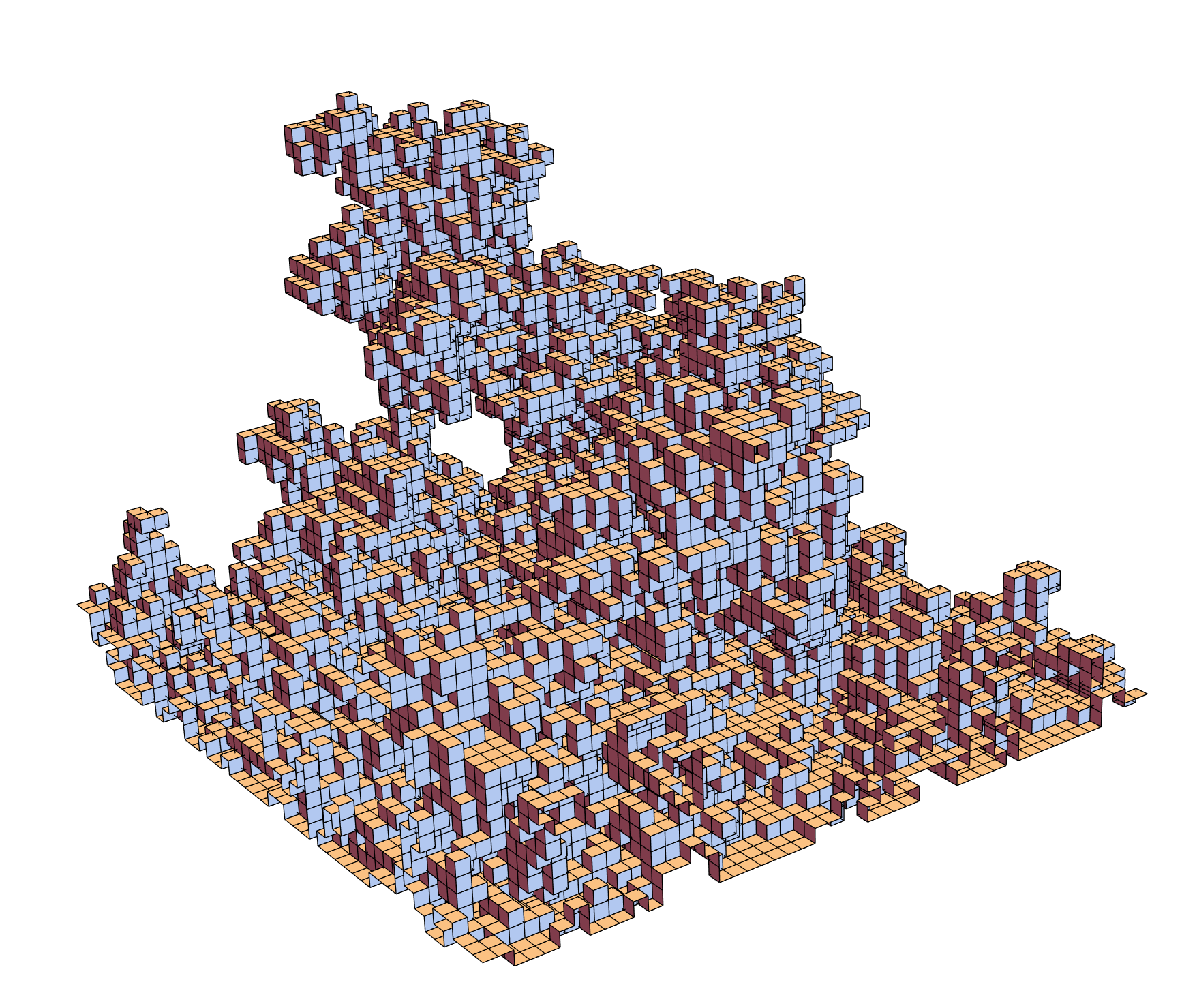}
	\includegraphics[height=0.4\textwidth]{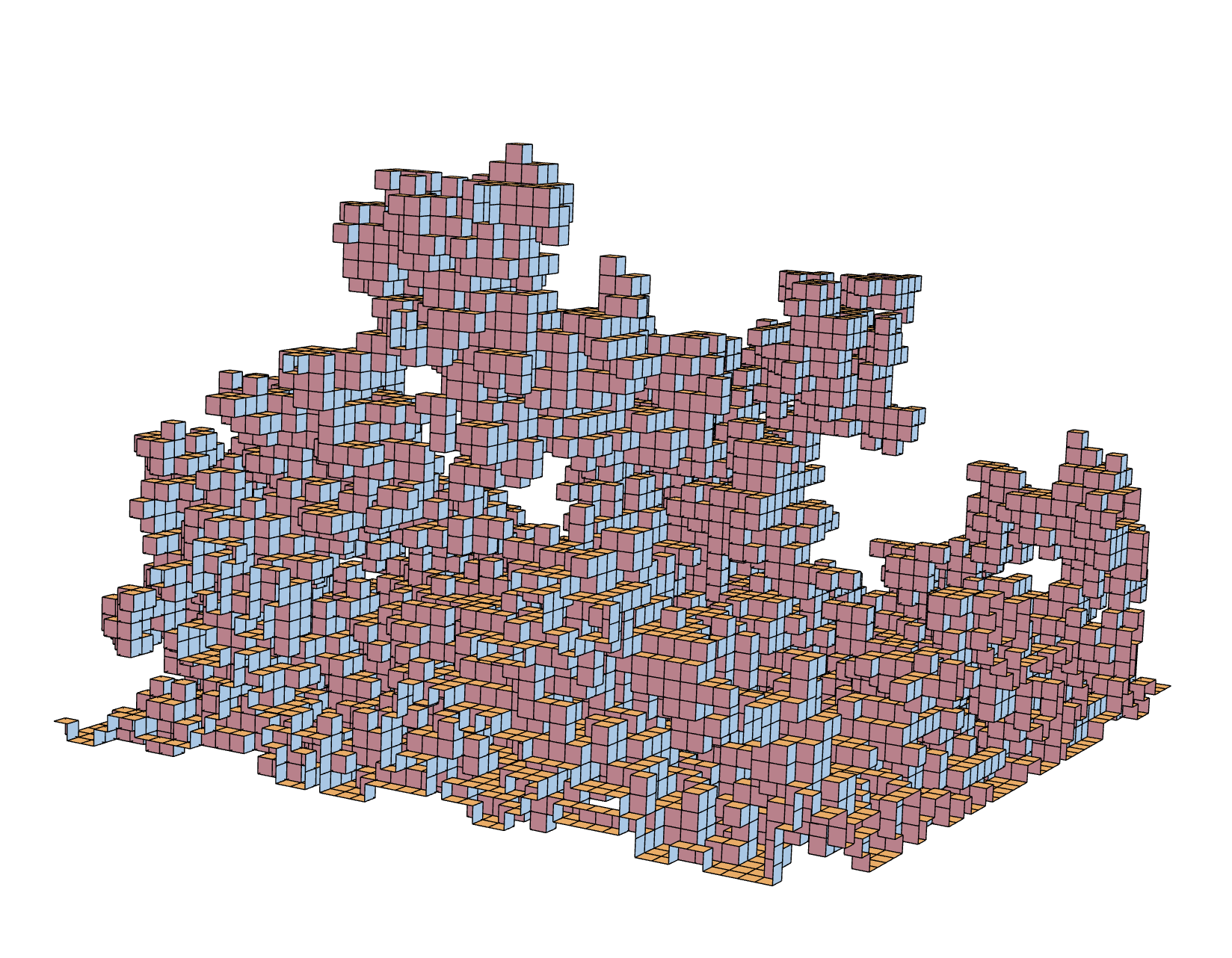}
    \caption{Two views of a plaquette crossing of $\Lambda_{25}$ in plaquette percolation with probability $p=1-0.246\approx 1-p_c\paren{\Z^3}+.003,$ chosen among the crossings with a minimal number of plaquettes. Note the presence of handles. The figure was created by Benjamin Atelsek, Gregory Maleski, and Tristan Napoliello as part of a project run through the Mason Experimental Geometry Lab (MEGL).}
\label{fig:plaquettecrossing}
\end{figure}

We also consider box crossings, in which we have analogous behavior in the same regimes.

\begin{prop}
Consider plaquette percolation on the cube $\Lambda_N=\brac{-N,N}^3$ and let $O$ be the event that there is a collection of plaquettes separating the top of the cube from the bottom of the cube that is homeomorphic to a disk. Then
$$\mathbb{P}_p\paren{O}\to\begin{cases} 0 & p<1-p_c^{\lim} \\ 1 & p>1-\pi_c^{e}\end{cases}$$
with high probability as $N\to\infty.$ 
\end{prop} 

The case $p>1-\pi_c^{e}$ follows from a straightforward application of Proposition~\ref{prop:duality} and a union bound. The case $p<1-p_c^{\lim}$ is a consequence of a result giving a more detailed view of the plaquette system when $1-p_c\paren{\Z^3}<p<1-p_c^{\lim}$, which we now describe.

We say that a set of plaquettes $P' $ is a plaquette crossing if $P'$ separates the top and bottom faces of $\Lambda_N$ and $P'\setminus \sigma$ does not for any plaquette $\sigma\in P'.$ For $p\in \left(1-p_c\paren{\Z^3},1-p_c^{\lim}\right)$, all plaquette crossings of a sufficiently large cube contain many handles. For a percolation subcomplex $P$ of $\Lambda_N$ containing no $2$-cells of $\partial \Lambda_N$ let 
$\mathcal{S}\paren{P}$ be the collection of plaquette crossings of $\Lambda_N$ contained in $P.$
Define $\mathrm{mh}_N\paren{P}$ be be $\infty$ if $\mathcal{S}\paren{P}=\varnothing$ and otherwise 
\[\mathrm{mh}_N\paren{P} \coloneqq \min_{S \in \mathcal{S}\paren{P}} \rank H_1\paren{S\cup \partial \Lambda_N}\,,\]
where $H_1\paren{X}$ is the first homology group of $X$ with coefficients in $\Z.$  Our main result in the intermediate regime gives the growth rate of $\mathrm{mh}_N\paren{P}$ up to a logarithmic factor.

\begin{theorem}\label{thm:handles}
Let $1-p_c\paren{\Z^3}<p<1-p_c^{\lim}.$ Then there exist $c_1\paren{p},c_2\paren{p} > 0$ so that 
\[c_1\paren{p} \frac{N^2}{\log N} < \mathrm{mh}_N\paren{P} < c_2 N^2\]
with high probability as $N \to \infty.$
\end{theorem}

\section{Outline}
In Section~\ref{sec:duality} we define the limited notion of entanglement that we are working with and prove Proposition~\ref{prop:vgammadual} relating the event $U_{\gamma}$ to the dual bond system. Along the way we discuss useful results on surfaces in $\R^3$ that allow us to avoid pathologies in Section~\ref{sec:subtleties} and provide tools to find spheres separating particular sets in Section~\ref{sec:sepspheres}. We then prove perimeter law and area law bounds in Sections~\ref{sec:perim} and~\ref{sec:area}. In both cases we roughly follow the arguments for analogous results in~\cite{ACCFR83}, though significant modifications are required in both settings. We also show that the critical probability for $m$-entangled percolation is continuous in slabs in Section~\ref{sec:slabs}. Finally, in Section~\ref{sec:genus} we prove Theorem~\ref{thm:handles} by showing that dual entangled crossings of smaller subboxes each create nontrivial homology in any plaquette crossing of a larger box.

\section{Limited Entanglement and Entanglement Duality}\label{sec:duality}
We now define $m$-entanglement more precisely: Given a graph $G,$ we will build a family of graphs $E(m) = E(G,m)$ with the same vertex set $V$ of $G.$ For $S\subset V,$ with $\abs{S}\leq m,$ include a complete subgraph on the vertices of $S$ if $S$ is entangled. Then $v_1,v_2 \in V$ are adjacent if there is an entangled subgraph of $G$ containing $v_1$ and $v_2$ with at most $m$ edges. Then we say that $G$ is $m$-entangled if $E(G,m)$ is connected. Notice that $m$-entangled percolation is $m$-dependent, or in other words the states of edges at $L^1$ distance greater than $m$ are independent.

For completeness, we verify that $m$-entanglement is a stronger condition than general entanglement.
\begin{lemma}
Let $m \geq 1.$ If $G$ is $m$-entangled, then $G$ is also entangled.
\end{lemma}

\begin{proof}
Supposed that $G$ is $m$-entangled and separated by a sphere $\mathcal{S}.$ Then $\mathcal{S}$ also separates the vertices of $E(G,l)$ into $A\cup B$ with $A,B \neq \emptyset,$ where the connected components corresponding to the vertices of $A$ are in the bounded component of $\R^3 \setminus \mathcal{S}$ and the components corresponding to the vertices of $B$ are in the unbounded component of $\R^3 \setminus \mathcal{S}.$ But then since $E(G,m)$ is connected, there must be an edge between a vertex of $A$ and a vertex of $B.$ Therefore, $\mathcal{S}$ separates an entangled set, a contradiction.
\end{proof}

Before we continue, we note that there are other possible limited notions of entanglement. Another option is to only allow entangled sets with a bounded number of connected components rather than a bounded number of edges. A third, considered by Atapour and Madras~\cite{atapour2010number} involves trying to fill in an entangled set by filling in each loop of bonds with a minimal surface. If the set is similar to a bracelet made of chain links, the resulting set will be connected after one step. However, if a larger bracelet is constructed out of \say{loops}, each of which is a bracelet of the previous type, a second round of fillings is required to connect the whole set. In this way, one can define a limited entanglement model by only allowing a bounded number of steps in this process. However, it remains an open question whether any of the three limited models are probabilistically equivalent to each other or to general entanglement in the limit.

\subsection{Topological Subtleties}\label{sec:subtleties}
In this section, we will deal in some detail with different types of mappings of spheres and disks into $\R^3.$ We pause here to provide some context and examples that motivate these distinctions. When we study the event $U_\gamma,$ we will want to characterize whether $\gamma$ is contractible in $P$ in terms of the topology of $P.$ The most straightforward way for $\gamma$ to be contractible is if it is the boundary of an embedded disk (recall that an embedding is a map that is a homeomorphism onto its image). However, it is not difficult to see that this is not a necessary condition: a ``pinched'' disk in which two points are identified also suffices. The correct equivalent condition is  that $\gamma$ is the boundary of a continuous image of a disk in $P.$ Another potentially troublesome issue with embeddings of spheres arises when trying to prove Proposition~\ref{prop:vgammadual}. There, we would like to be able to contract any loop in the complement of an embedded sphere, but the Alexander horned sphere~\cite{alexander1924example} provides an example of a pathological embedding where this is not necessarily possible. That is, the Jordan--Schoenflies Theorem that a simple closed curve embedded in the the two-dimensional sphere splits it into two components homeomorphic to two-dimensional balls does not generalize to three dimensions without further hypotheses.

Fortunately, these complications can be avoided in the current context: $B$ is a closed subset of $\R^3$, which affords us some wiggle room to find ``nice'' embeddings of spheres and disks in $\R^3\setminus B.$  Recall that a piecewise--linear embedding of a surface is an embedding whose image is a finite union of simplices that meet nicely at their boundaries (that is, a collection of vertices, edges, and triangular faces so that two edges can meet only at a single vertex and two faces can meet either at a single vertex or at a single edge). A piecewise--linear embedded curve is defined similarly, but with only vertices and edges.

We now state three fundamental results in the topology of $3$-manifolds, two of which allow us to use piecewise--linear embeddings to avoid these issues.

\begin{theorem}[Dehn's Lemma~\cite{papakyriakopoulos1957dehn}]\label{lemma:dehn}
    Let $\gamma$ be a piecewise--linear embedded curve in a $3$-manifold $M.$ Then if there is a continuous image $\mathcal{I} \subset M$ of a disk so that $\gamma = \partial \mathcal{I},$ then there is also a piecewise--linear embedded disk $\mathcal{D} \subset M$ so that $\gamma = \partial \mathcal{D}.$
\end{theorem}

\begin{theorem}[Bing~\cite{bing1957approximating}]\label{theorem:bingapprox}
    Let $\mathcal{S} \subset \R^3$ be an embedded sphere. Then for any $\epsilon > 0$ there is a piecewise--linear embedded sphere $\mathcal{S}'$ and a homeomorphism $f : \mathcal{S} \to \mathcal{S}'$ so that for any $x \in \mathcal{S},$ $d\paren{x,f\paren{x}} < \epsilon.$
\end{theorem}

\begin{theorem}[Alexander~\cite{alexander1924subdivision}]\label{theorem:alexandersimply}
    Let $\mathcal{S}$ be a piecewise--linear embedded two-sphere in the three-sphere $S^3=\R^3\cup\infty.$ Then both connected components of $S^3 \setminus \mathcal{S}$ are homeomorphic to three-dimensional balls.
\end{theorem}

The following is a straightforward consequence of Theorem~\ref{theorem:bingapprox}.
\begin{corollary}\label{cor:equivalence}
    Let $E \subset \Z^3 \cup \paren{\Z^3}^*$ be a set of bonds and dual bonds. Then $E$ is separated by an embedded sphere if and only if it is separated by a piecewise--linear embedded sphere.
\end{corollary}
\begin{proof}
Let $\mathcal{S}$ be an embedding of a sphere in $\R^3\setminus E.$  $\mathcal{S}$ is compact and $E$ is closed, so there exists $\epsilon>0$ so that the closed $\epsilon$-neighborhood $\mathcal{S}_{\epsilon}=\set{x\in\R^3:d\paren{x,\mathcal{S}}\leq \epsilon}$ is disjoint from $E$. Then, by Theorem~\ref{theorem:bingapprox} there is a piecewise--linear embedded sphere $\mathcal{S}'\subset \R^3$ and a homeomorphism $f:\mathcal{S}\to\mathcal{S}'$ so that $d\paren{x,f\paren{x}}<\epsilon$ for all $x\in \mathcal{S}.$ Observe that $\mathcal{S}, \mathcal{S}_{\epsilon},$ and $\mathcal{S}'$ all induce the same partition of $E.$ 
\end{proof}

As a result, we henceforth assume that all embedded spheres that we consider are piecewise--linear. 

We also will need the notion of transversality. Recall that manifolds $X,Y \subset \R^3$ intersect transversely if at any point $v \in X \cap Y,$ the union of the tangent spaces $T_v X$ and $T_v Y$ generates $\R^3.$ In the case of a piecewise--linear manifold, we define the tangent space of a point at the intersection of more than one linear piece to be the intersection of the tangent spaces of the pieces. It is well known that two piecewise--linear manifolds can be perturbed an arbitrarily small amount to intersect transversely, and that the resulting transverse intersection is also a manifold~\cite{armstrong1967transversality}. For example, if a piecewise--linear embedded sphere intersects the boundary of a box transversely, the intersection is a one manifold: a disjoint union of simple closed curves. 

Now we prove a finite volume duality statement relating the event that there is an entangled crossing in $B$ of the top and bottom of the box to the event that a nontrivial loop in the remaining sides is contractible in $P.$ However, we need to be careful here too about the correct notion of entangled crossing since a na\"{i}ve version does not have the correct duality properties.

\begin{figure}
	\includegraphics[height=0.4\textwidth]{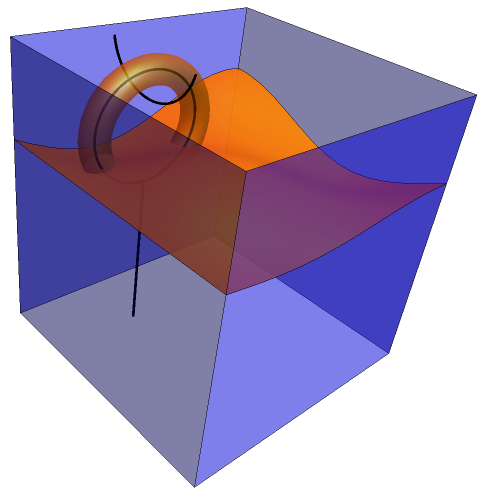}\qquad\qquad
	\includegraphics[height=0.4\textwidth]{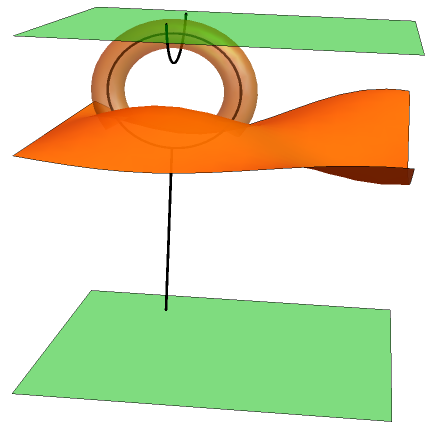}
    \caption{An illustration of the event in Proposition \ref{prop:duality}, shown from two different perspectives. $D_2$ is shown in blue on the left and $D_1$ in green on the right.  Observe that $\pi_1\paren{D_2}=\Z$ and that the generator of that fundamental group is not contractible in the union of $D_2$ and the orange surface due to the presence of a handle. The black paths are an entangled crossing of $D_1.$}
\label{fig:boxentanglement}
\end{figure}

More precisely, let $$R = [0,a]\times [0,b] \times [0,c]$$ for some $a,b,c \in \N,$ let 
\begin{equation}\label{eq:Rboundary}
\begin{aligned}
    &D_1^- = D_1^-(R) = \brac{0, a} \times [0,b] \times \set{1/2},\\
    &D_1^+ = D_1^+(R) = \brac{0, a} \times [0,b] \times \set{c-1/2},\\
    &D_1 = D_1(R) = D_1^- \cup D_1^+, \\
    &D_2 = D_2(R) = \partial R \cap \set{1/2 \leq z \leq c-1/2}.
\end{aligned}
\end{equation}
Let $P_R$ be the union of $P \cap R \cap \set{1 \leq z \leq c-1}$ and all plaquettes of $D_2,$ and let $B_R$ be the union of $B \cap \brac{1/2,a-1/2} \times \brac{1/2,b-1/2} \times \brac{1/2,c-1/2}$ and all of the bonds of $D_1.$ These sets should be thought of in terms of boundary conditions on the box; in the case of $P_R$ we are wiring the four faces of $D_2$ together and leaving the remaining two faces free and in the case of $B_R$ we do the opposite, with adjustments to fit the dual lattice instead of the primal one.

In order to obtain a dual condition for a disk crossing, we need to further restrict the possible separating spheres by extending a component of $D_1$ to a plane. We say that $B_R$ contains an entangled crossing if there is an entangled set of bonds $E \subset B_R$ that intersects both $D_1^-$ and $D_1^+.$

Given a subset $X \subset \R^3,$ we define $I\paren{X}$ to be the union of the bounded components of $\R^3 \setminus X.$

\begin{figure}
	\includegraphics[height=0.4\textwidth]{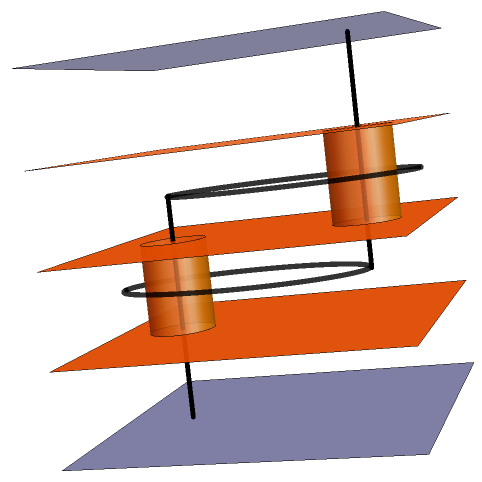}\qquad\qquad
    \includegraphics[height=0.4\textwidth]{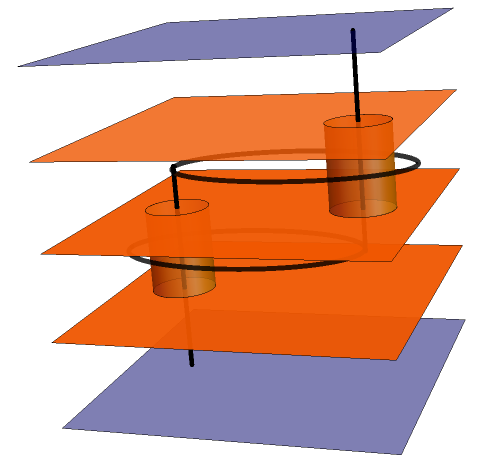}
    \caption{An example inspired by Bing's ``House with Two Rooms''\cite{bing1964some}, shown from two different perspectives. Here, $P_R$ is depicted in orange, $D_1$ is illustrated in blue, and a system of bonds homotopy equivalent to $B_R$ is drawn in black. $P_R$ is formed from three parallel rectangles by removing a disk from the top and bottom rectangles and two disks from the middle rectangle and connecting the middle rectangle to the other two by cylinders. It is easy to see that $D_1$ is separated by a sphere in the complement of $B_R,$ but it is more difficult to visualize a disk inside of $R\setminus B_R$ separating $D_1.$ If $\alpha_1, \alpha_2,$ and $\alpha_3$ are the three rectangular loops in $D_2\cap P_R,$ ordered from top to bottom, then $\brac{\alpha_2}=\brac{\alpha_1}\ast\brac{\alpha_3}$ in $\pi_1\paren{P_R},$ but $\brac{\alpha_1}=\brac{\alpha_2}=\brac{\alpha_3}$ in $\pi_1\paren{D_2}.$ It follows that $\brac{\alpha_1}=e$ in  $\pi_1\paren{P_R\cup D_2}.$ Then, Dehn's lemma implies that $\alpha_1$ is the boundary of a disk in a small neighborhood of $D_2\cup P_R,$ which can then be pushed inside of $R\setminus B_R.$ This example will also be relevant in Section~\ref{sec:genus}.}
\label{fig:tworooms}
\end{figure}

The following proposition is illustrated in Figures~\ref{fig:boxentanglement} and~\ref{fig:tworooms}, and is essentially a consequence of Theorem~\ref{lemma:dehn}, adapted to our setting. In order to define one of the equivalent conditions, recall that an inclusion of spaces $X \hookrightarrow Y$ induces a map on fundamental groups $\pi_1\paren{X} \to \pi_1\paren{Y}$ that maps an equivalence class of loops to the equivalence class of the image under the inclusion of one of its representatives.

\begin{prop}\label{prop:duality}
The following are equivalent:
\begin{enumerate}
    \item There is no piecewise--linear embedded disk in $R \setminus B_R$ that separates $D_1.$
    \item The map $\pi_1(D_2) \mapsto \pi_1(P_R)$ induced by inclusion is nontrivial.
    \item There is no piecewise--linear embedded sphere in $\R^3\setminus B_R$ that separates $D_1.$ 
    \item $B_R$ contains an entangled crossing of $D_1.$
\end{enumerate}
In addition, we show that if this event does not occur then there is a surface of plaquettes separating the two components of $D_1$ that is the continuous image of a disk. 
\end{prop}

\begin{proof}
The fact that $(1)$ and $(2)$ are equivalent is an immediate consequence of Theorem~\ref{lemma:dehn} and the fact that $R \setminus B_R$ deformation retracts onto $P_R$ with the small technical caveat that the disk found this way may need to be deformed slightly in order to ensure that it is within $R.$ In addition, $(3)$ and $(4)$ are equivalent by Corollary~\ref{cor:equivalence}.

We now prove that $(2)$ and $(3)$ are equivalent.

$(2) \implies (3):$\\
Suppose that there is a piecewise--linear embedded sphere $\mathcal{S}$ in $\R^3\setminus B_R$ such that $D_1^-$ and $D_1^+$ are contained in the bounded and unbounded components respectively of $\R^3 \setminus \mathcal{S}.$ 
Since $B_R \cap \set{z \in \set{0,c}} = \emptyset,$ we may assume that $\mathcal{S} \cap \partial R \subset D_2.$ Since $B_R$ is closed, we may  also assume that $\mathcal{S}$ intersects $D_2$ transversely, so $\mathcal{S} \cap D_2$ is a finite union of simple loops. We modify $\mathcal{S}$ to remove loops that are trivial in $\pi_1\paren{D_2}.$ If $\mathcal{S} \cap \partial R$ contains a loop that is contractible in $D_2$ then it contains at least one such loop $\gamma$ that bounds a disk $\mathcal{D}_{\gamma}$ in $D_2.$ Then the union of $\mathcal{D}_{\gamma}$ and one of the components of $\mathcal{S} \setminus \gamma$ is a sphere that separates $D_1^-$ and $D_1^+.$ Replace $\mathcal{S}$ by this sphere and push $\mathcal{D}_{\gamma}$ into $R$ to reduce the number of components of $\partial R\cap \mathcal{S}$ by one. By applying this process to each contractible loop, we may assume without loss of generality that no loop of $\mathcal{S} \cap D_2$ is contractible in $D_2.$ 

Since $\mathcal{S} \cap \partial R$ is a union of disjoint loops, $\mathcal{S} \cap R$ is a disjoint union of surfaces with boundary. We first consider the case in which there is such a surface $\mathcal{T}$ with more than two boundary components, such as the example illustrated in Figure~\ref{fig:tworooms}. Let $\beta_1,\ldots,\beta_k$ be the boundary components of $\mathcal{T}.$ Orient them so that they are all homotopic in $\pi_1\paren{D_2}$ to a generator $\alpha$ of that group. Since $\mathcal{T}$ is homeomorphic to a two-sphere with $k$ disks removed, we have a relation
\begin{equation}\label{eq:pantssum}
    \brac{\beta_1}^{a_1}\ast \ldots\ast \brac{\beta_k}^{a_k}=e\in \pi_1\paren{\mathcal{T}}
\end{equation}
where $a_1,\ldots,a_k\in\set{-1,1}.$ In the next paragraph, we show that $a_1,\ldots,a_k$ must all be $1$ or $-1.$  As such, under the inclusion $i : \pi_1\paren{D_2} \to \pi_1\paren{D_2 \cup \mathcal{T}}$ we have that 
\[i\paren{\brac{\alpha}} = i\paren{\brac{\alpha}}^{k-1}\]
so either $\brac{\alpha}=e$ in $\pi_1\paren{D_2\cup \mathcal{T}}$ (and we are done) or $\brac{\alpha}$ is an element of finite order in $\pi_1\paren{D_2\cup \mathcal{T}}.$ However, $D_2 \cup \mathcal{T}$ is a locally path--connected and semilocally simply connected subset of $\R^3,$ so it cannot have a nontrivial element of finite order~\cite{4034711}.

We return to check that the boundary loops of $\mathcal{T}$ are oriented in the same direction. Order the boundary components $\beta_j$ by height in $D_2.$ Take inverses in~(\ref{eq:pantssum}) if necessary so that $a_1=1$ and let $j$ be the smallest index so that $a_{j+1}=-1.$ $\beta_j$ and $\beta_{j+1}$ bound a cylinder $\mathcal{C}_1$ in $D_2$ which does not intersect any other components of $\partial \mathcal{T}.$ Let $\mathcal{C}_2$ be the cylinder obtained by attaching disjoint disks outside of $R$ to the remaining $k-2$ loops in the boundary of $\mathcal{T}$  (which is possible because $\beta_j$ and $\beta_{j+1}$ are adjacent). We have that $\brac{\beta_j}=\brac{\beta_{j+1}}$ in $\pi_1\paren{\mathcal{C}_1}$ but $\brac{\beta_j}=\brac{\beta_{j+1}}^{-1}$ in $\pi_1\paren{\mathcal{C}_2}.$ Thus $\brac{\beta_{j+1}}$ represents an element of order two in the closed surface that is the union of $\mathcal{C}_1$ and $\mathcal{C}_2,$ so that surface is a Klein bottle. This is a contradiction, as the Klein bottle does not embed in $\R^3.$ 

We now turn to the case in which all components of $\mathcal{S} \cap R$ have at most two boundary components, meaning that each is either a cylinder or a disk. Since $B_R$ does not contain any bonds above $D_1^+$ or below $D_1^-$ we may assume that $\mathcal{S}$ does not intersect $R\cap\set{z\geq c-1/2}$ or $R\cap\set{z\leq 1/2}.$ By construction, $\partial R\setminus \mathcal{S}$ is a disjoint union of finitely many cylinders and two disks, one containing the top face of $R$ and one containing the bottom face of $R.$  Since $\mathcal{S}$ contains $D_1^-$ in its interior,  $I\paren{\mathcal{S}}\cap \partial R$ must include the bottom disk but not the top disk.
Let 
$$
\mathcal{S}'' \coloneqq \partial \paren{I\paren{\mathcal{S}}\cap R}\,.
$$
Then by construction, $\mathcal{S}''$ is a (not necessarily connected) closed oriented surface obtained by gluing together cylinders and disks each contained in either $I\paren{\mathcal{S}}\cap \partial R$  or $\mathcal{S}\cap I\paren{\partial R}$ along their boundaries in $\mathcal{S}\cap \partial R.$ Let $\mathcal{S}' \subset \mathcal{S}''$ be the component of $\mathcal{S}''$ containing the bottom disk of $\partial R\setminus \mathcal{S}$ so $\mathcal{S}'$ contains $D_1^-$ in its interior. Since attaching a cylinder to a disk creates another disk, to create a closed surface there must eventually be another disk $D\subset R\setminus B_R$ attached, making $\mathcal{S}'$ a sphere. Since the top disk of $\partial R\setminus \mathcal{S}$ is not included in $I\paren{\mathcal{S}}\cap \partial R,$ $D$ must be contained in $\mathcal{S}\cap I\paren{\partial R}.$ Thus $D$ is a disk contained inside $R\setminus B_R$ separating $D_1.$

$(3) \implies (2):$\\
Suppose the inclusion $\pi_1(D_2) \mapsto \pi_1(D_2 \cup P_R)$ is trivial. Then the loop $\gamma_1 \coloneqq \partial R \cap \set{z = 1/2}$ is contractible in $D_2 \cup P_R.$ Let $\mathbb{D}^2$ denote the 2-dimensional disk. Then the contraction of $\gamma_1$ induces a map $i_1 : \mathbb{D}^2 \to P_R$ such that $i_1\paren{\partial \mathbb{D}^2} =  \gamma_1.$ By Theorem~\ref{lemma:dehn} and the fact that there is an open neighborhood of $P_R$ in $R \setminus B_R,$ there is then an embedding $i_1' : \mathcal{D}^2 \to R \setminus B_R$ such that $i_1'\paren{\partial \mathbb{D}^2} =  \gamma_1.$ Of course we can also find an embedding $i_2 : \mathbb{D}^2 \to \set{z \leq 1/2} \setminus R$ such that $i_2\paren{\partial \mathbb{D}^2} = \gamma_1.$ Together these give us an embedding $j : \mathbb{S}^2 \to \mathbb{R}^3,$ whose image separates $D_1$ as desired.
\end{proof}

\subsection{Separating Spheres}\label{sec:sepspheres}

We now work towards a characterization of the event $U_{\gamma}$ in terms of the dual bond system. Recall that the entangled component of a vertex $v$ in a graph $G$ is its connected component after adding long-range edges between vertices in the same entangled set, or equivalently the union of the vertices in entangled sets containing $v.$ Note that this is consistent with the wired notion of entanglement as discussed in~\cite{grimmett2000entanglement}, since we are allowing infinite sets of bonds. It will be useful to have spheres that enclose some entangled components but not others. The existence of such spheres does not follow immediately from the definition of entanglement, but can be proven as a straightforward lemma. We begin with the finite case.  

\begin{lemma}\label{lemma:components}
Let $E$ and $F$ be finite unions of finite entangled components in $E \cup F$ so that $E\cap F=\varnothing.$ Then there exists a sphere $\mathcal{S}$ separating $E$ and $F$ such that $E$ is contained in the bounded component of $\mathbb{R}^3 \setminus \mathcal{S}.$
\end{lemma}

\begin{proof}
First, suppose that we have two finite sets of bonds $G$ and $H$ separated by a sphere $\mathcal{S}_G$ with $G$ in the bounded component of its complement. Then we can construct a sphere $\mathcal{S}_H$ separating $G$ and $H$ with $H$ in its bounded component in the following way. Let $\mathcal{S}_{(G,H)}$ be a large sphere containing both $\mathcal{S}_G$ and $H$. Let $\xi$ be a smooth path connecting a point $x_G \in \mathcal{S}_G$ with a point $x_{(G,H)} \in \mathcal{S}_{(G,H)}$ that does not intersect $H.$ Since $H$ is closed, we may thicken $\xi$ into a tube connecting disks about $x_G$ and $x_{(G,H)}$ that also does not intersect $H.$ Removing the disks from the spheres and attaching them with the tube creates the desired embedded sphere.

We now consider the general case.  It is enough to find a sphere separating a single entangled component $A$ of $E$ from the rest because we can connect spheres containing the desired components by tubes as above. If $E$ is not entangled, there is a sphere $\mathcal{S}_1$ that separates $E$ into $E_1$ and $E_1'.$ Suppose without loss of generality that $A \subset E_1.$ By the above argument we can assume that $E_1$ is in the bounded component of $\R^3 \setminus \mathcal{S}_1.$ If $E_1 = A,$ then we are done. Otherwise, $E_1$ is not entangled, so there is a sphere $\mathcal{S}_2$ that separates it into $E_2$ and $E_2'.$ A priori $\mathcal{S}_2$ may intersect $E_1'.$ However, $I\paren{\mathcal{S}_1}$ is homeomorphic to $\R^3$, so we may assume that $\mathcal{S}_2\subset I\paren{\mathcal{S}_1}.$ Since $E$ is a finite union of components, repeated application of this process will terminate, giving us the desired sphere separating $A$ from $E \setminus A.$
\end{proof}

In fact, the previous lemma can be improved to remove the requirement that both sets be finite.

\begin{lemma}\label{lemma:separation}
Let $E$ a finite union of finite entangled components of $B.$ Then there is a sphere of $\mathbb{R}^3$ that separates $E$ and $B \setminus E.$
\end{lemma}

\begin{proof}
As in Lemma~\ref{lemma:components}, spheres containing finite components can be joined by tubes to create a single sphere, so it is enough to consider the case where $E$ is a single entangled component. For convenience shift the dual bonds to $\mathbb{Z}^3$ and assume that $0 \in E.$ Let $Q_n = \{(x,y,z) \in \mathbb{R}^3 : -n \leq x,y,z \leq n\}.$ 

We claim that it suffices to show that if $E$ is finite then there is a sphere $\mathcal{S}$ of $\R^3\setminus B$ containing $E$ in its interior. Assuming this, there is a homeomorphism $f$ from the interior $I\paren{\mathcal{S}}$ to $\R^3$; applying Lemma~\ref{lemma:components} to the finite graph $f\paren{B\cap I\paren{\mathcal{S}}}$ yields a sphere $\mathcal{S}'\subset I\paren{\mathcal{S}}\setminus B$ containing $E$ in its interior and all vertices of $I\paren{\mathcal{S}}\setminus E$ in its exterior.

We now show that if $E$ is finite then there is a sphere $\mathcal{S}$ of $\R^3\setminus B$ containing $E$ in its interior. This is Proposition 2.4 in~\cite{grimmett2000entanglement}, for which we give a shorter proof. Assume that the entangled component $E$ of the origin is not contained in the interior of any sphere of $\R^3\setminus B.$ We will show that $E$ is infinite. Let $T_n$ be the set of bonds in the boundary of the cube $Q_n=\brac{-n,n}^3.$ For any for any $n>0$ there must be a graph $E_n\subset Q_n\cap B$ containing the origin so that $E_n\cup T_n$ is entangled in $Q_n \cap B$. We will show the contrapositive. Suppose there is a sphere $\mathcal{S}\subset \R^3\setminus Q_n\cap B$ separating $T_n$ from $0.$ $\mathcal{S} \cap \partial Q_n$ is a union of loops, each within a unit square of bonds of $T_n.$ We may iteratively fill these loops with disks, pushing the resulting spheres locally into or out of $Q_n.$ This yields a collection of spheres disjoint from $\partial Q_n,$ one of which must separate $0$ from $T_n.$ That sphere must be contained inside $Q_n$ and thus must be disjoint from $B$ and contain $0$ in its interior. 

By the same reasoning, for $m>n,$ $\paren{E_m\cap Q_n} \cup T_n$ is entangled in $Q_n \cap B$. There are only finitely many subgraphs of $Q_n$ so one $\hat{E}_n$ must occur as $E_m\cap Q_n$ for infinitely many $m$ As such, there is a collection of graphs $\hat{E}_{1}\subset \hat{E}_2\ldots$ with that property. Then $\cup_{n}\hat{E}_n$ is an infinite entangled graph containing the origin. 
\end{proof}

The same proof works in a slightly more general setting, which we state separately for the upcoming proof and for later convenience.

\begin{corollary}\label{cor:dualseparation}
    Let $F \subset \Z^3 \cup \paren{\Z^3}^*$ be a set of bonds and dual bonds and let $E \subset F$ be a finite union of finite entangled components of $F.$ Then there is a sphere of $\mathbb{R}^3$ that separates $E$ and $F \setminus E.$
\end{corollary}

We can now prove Proposition~\ref{prop:vgammadual}.

\begin{proof}[Proof of Proposition~\ref{prop:vgammadual}]
    Suppose $U_{\gamma}$ occurs. Then there is an $N$ and a disk image $\mathcal{I}_{\gamma} \subset P\cap \Lambda_N$ so that $\gamma = \partial \mathcal{I}_{\gamma}.$ In particular, $U_{\gamma}$ occurs within $\Lambda_{N+1}$ with free boundary conditions, meaning that we include none of the plaquettes in $\partial \Lambda_{N+1}.$ Then by Theorem~\ref{lemma:dehn} there is a set $\mathcal{D} \subset \Lambda_{N+1} \setminus \paren{B \cup \partial \Lambda_{N+1}}$ that is homeomorphic to a disk so that $\gamma = \partial \mathcal{D}.$ Since $B \cup \partial \Lambda_{N+1}$ is closed, we can thicken $\mathcal{D}$ to obtain a set $W$ so that $\mathcal{D}\subset W\subset \Lambda_{N+1} \setminus \paren{B \cup \partial \Lambda_{N+1}}$ and so that $\partial W$ is homeomorphic to a sphere. Then $\partial W$ separates $\gamma$ and $B$ so $\gamma$ is not entangled with any subset of edges of $B.$

     Now suppose $\gamma$ is not entangled with any subset of edges of $B.$ Then by Corollaries~\ref{cor:dualseparation} and~\ref{cor:equivalence}, there is a piecewise--linear embedded sphere $\mathcal{S}$ that contains $\gamma$ in its interior and $B$ in its exterior. Then by Theorem~\ref{theorem:alexandersimply}, $\gamma$ is contractible in its component of $\R^3 \setminus \mathcal{S}$, so in particular it is contractible in $\R^3\setminus B.$ Since $\R^3\setminus B$ deformation retracts onto $P,$ $\gamma$ is contractible in $P.$  
\end{proof}

\section{Perimeter Law}\label{sec:perim}

We now consider the perimeter law regime. Let $p > 1-\pi_c^e.$ Notice that $\mathbb{P}_p(U_\gamma)$ is trivially bounded above by the probability that every edge of $\gamma$ is contained in at least one plaquette of $P$, which decays exponentially in the perimeter of $\gamma.$ 

Our strategy to prove a matching lower bound is similar to the argument given in~\cite{ACCFR83}: we prove that $\gamma$ is contractible by constructing an obstacle to an entangled set of dual bonds that interlocks $\gamma.$ Specifically, we build a doubly infinite cylinder passing through the area enclosed by $\gamma$ in its plane, and then show that the walls of the cylinder are disentangled from that plane with the appropriate probability. 

For two subsets of $G,H \subset \R^3$ we write $G \xleftrightarrow[]{e} H$ if there is an entangled subgraph that intersects both $G$ and $H.$

First we show that the positive dual $z$-axis is disentangled from the $xy$-plane with positive probability. Let 
\[K^+ = \{(1/2,1/2,z+1/2):z\in \mathbb{Z}^{\geq 0}\}\,,\] 
\[K_h^+ = K^+ \cap \{z \geq 1/2 + h\}\,,\] 
and 
\[G = \set{z = 0}\,.\] 
Notice that $G$ contains no dual vertices. 

\begin{lemma}
If $p > 1-\pi_c^e$, then $\mathbb{P}(K^+ \xleftrightarrow[]{e} G) < 1.$
\end{lemma}

\begin{proof}
Since any edge passing through $G$ must necessarily have an endpoint in $J \coloneqq (\Z^3)^* \cap \set{z = -1/2},$ we have 
\[\set{K^+ \xleftrightarrow[]{e} G} = \set{K^+ \xleftrightarrow[]{e} J}\,,\]
so we will work with the latter event. Then observe that
\[\sum_{x \in K^+}\sum_{y \in J} \mathbb{P}(x \xleftrightarrow[]{e} y) \leq \mathbb{E}(|C_0^e|) < \infty
\]
because $p > 1-\pi_c^e.$ Then
\begin{align*}
    \mathbb{P}(K_h^+ \xleftrightarrow[]{e} J) \leq \sum_{x \in K_h^+}\sum_{y \in J} \mathbb{P}(x \xleftrightarrow[]{e} y) \xrightarrow[]{h \to \infty} 0\,,
\end{align*}
since the sum is the tail of a convergent series. In particular, there is an $h_0 \in \N$ such that $\mathbb{P}(K_{h_0}^+ \xleftrightarrow[]{e} J) < 1.$ Now setting all outgoing bonds to be closed is clearly sufficient to isolate a finite set of sites, so $\mathbb{P}((K^+ \setminus K_{h_0}^+) \centernot{\xleftrightarrow[]{e}} J) \geq p^{4h_0+2}.$ Thus, using the FKG inequality we have
\begin{align*}
    \mathbb{P}(K^+ \centernot{\xleftrightarrow[]{e}} J) &\geq \mathbb{P}(K_{h_0}^+ \centernot{\xleftrightarrow[]{e}} J) \mathbb{P}((K^+ \setminus K_{h_0}^+) \centernot{\xleftrightarrow[]{e}} J)\\
    &\geq (1-\mathbb{P}(K_{h_0}^+ \xleftrightarrow[]{e} J)) \cdot p^{4h+2} > 0.
\end{align*}
\end{proof}

As an immediate corollary of the FKG inequality, we can replace the positive $z$ axis with the entire $z$ axis. Let 
\[K \coloneqq \set{\paren{1/2,1/2,1/2 + z} : z \in \Z }\,.\]

\begin{corollary}\label{cor:zaxis}
    If $p > 1-\pi_c^e$, then $\mathbb{P}\paren{K \xleftrightarrow[]{e} G} < 1\,.$
\end{corollary}

Now we will build the cylinder passing through $\gamma$ using several translated copies of $K.$ Without loss of generality, assume that $\gamma$ is contained in the plane $\set{z=0}.$ Let $\sigma_1,\ldots, \sigma_N$ be an enumeration of the plaquettes in $\set{z = 0}$ that both have a boundary edge in $\gamma$ and are in the bounded component of $\{z=0\} \setminus \gamma.$ In other words, these are the plaquettes along the inside of $\gamma$ in the plane $\set{z=0}.$ For each $1 \leq i \leq N,$ let $a_i$ be the center of $\sigma_i,$ and let $K^i \coloneqq K + a_i.$ The following lemma involves entanglement of a set of dual bonds with $\gamma,$ which we have not strictly speaking defined previously, since $\gamma$ is in the primal lattice. As in our earlier definition for subsets of $(\Z^3)^*$, we say a subset of $\R^3$ is entangled if it cannot be separated by an embedded sphere.

\begin{lemma}\label{lemma:perimeter}
Let $F_{\gamma}$ be the event that $B \cup \gamma$ has no infinite entangled component and $D_{\gamma} = \bigcap_{i \leq N} \set{K^i \centernot{\xleftrightarrow[]{e}} G}.$ Then $D_{\gamma} \cap F_{\gamma} \subset U_{\gamma}.$
\end{lemma}

\begin{proof}

Assume that $D_{\gamma}$ occurs. Let $E\subset B$ and let $E_1,\ldots, E_n$ be the the maximal entangled components of $E.$ We will show that $E\cup \gamma$ is not entangled. On the event $D_\gamma$, if the projection of $E_j$ onto $\{z=0\}$ intersects $\gamma$, then $E_j$ can either cross over $\gamma$ or under $\gamma$ but not both. In addition, if $E_j$ crosses over $\gamma$ and $E_k$ crosses under $\gamma$ then $E_j$ is entirely above $E_k$ in the sense that the $z$-coordinate of any point of $E_j$ is greater than the $z$-coordinate of any point of $E_k.$ Recall that a link diagram is a representation of a link as a 4-valent plane graph where at each vertex $v$ (which represents a crossing in the link), we mark two incident edges as crossing over at $v$ and two incident edges as crossing under at $v.$ Consider the link diagrams of $E$ and $E\cup \gamma$ obtained by projecting to the $xy$-plane (where we perturb each vertex and vertical bond to obtain a locally injective projection that is injective at the images of the vertices, without changing how the bonds overlap with $\gamma$). A foundational result of knot theory states that there is an ambient isotopy between two links if and only if a link diagram for one can be brought to a link diagram for the other by a finite sequence of moves called Reidemeister moves. There are three types of moves for links, and two additional moves for spatial graphs that involve vertices~\cite{kauffman1989invariants}. The five moves for spatial graphs are called generalized Reidemeister moves. 

It is not difficult to check that two spatial graphs are separable by spheres if and only if they can be separated by an ambient isotopy. To see the forward direction, take a separating sphere and use an isotopy of $\R^3$ that contracts the interior of the sphere to a small neighborhood of a point, and then use an isotopy that moves that neighborhood away from the other subgraph. As such, in the present argument, there is a finite sequence of generalized Reidemeister moves that turn $E$ into $n$ disjoint link diagrams, one for each $E_j.$ 
 
\begin{figure}
    \centering
    \includegraphics[width=0.95\linewidth]{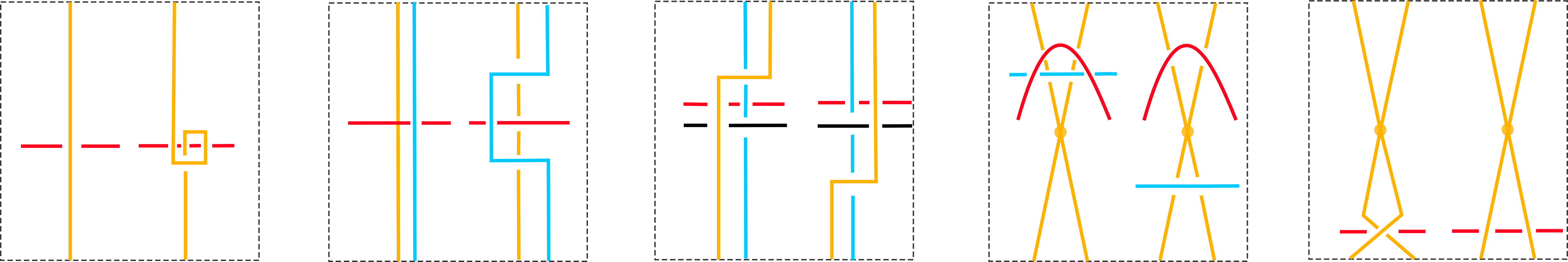}
    \caption{An illustration of each of the Reidemeister moves with the addition of a strand from $\gamma$ in red in several possible configurations.}
    \label{fig:reidemeister}
\end{figure}

We show that each of the generalized Reidemeister moves performed on the link diagram of $E$ can be realized by an ambient isotopy of $E\cup \gamma.$   We verify both that the move can be performed and that it preserves the property that each entangled component of $E$ either does not overlap with $\gamma$ or crosses entirely over or below $\gamma.$

Each of the five generalized Reidemeister moves can be realized by an ambient isotopy which moves a single strand $\alpha$, one which crosses entirely above or below the other strands. For example, in the Type III move the bottom strand can be moved down until it is completely below the other two strands and then moved to the other side of the crossing. So suppose that $\gamma$ crosses one or more of the strands in the region in which a Reidemeister move is defined and let $\alpha$ be as before. Without loss of generality, suppose that $\alpha$ crosses above the other strands in the region.  If the entangled component of $\alpha$ does not overlap with $\gamma$ or if it crosses above $\gamma$ we can isotope $\alpha$ without moving its endpoints so that an interior segment with all of the crossings is contained in a halfspace entirely above $\gamma$ and the other strands in the region. Then we can move this interior segment in the halfspace to realize the move. If $\alpha$ crosses below $\gamma,$ then $\gamma$ either does not cross or crosses above the remaining segments. As such, we can move an interior segment of $\gamma$ so that it is completely above the domain of the ambient isotopy realizing the move for the link diagram of $E.$ We then realize the move, which we note preserves the property that $\alpha$ only crosses below $\gamma.$

Now since any diagram of the link realized by $E$ can be reached from any other by a sequence of Reidemeister moves and $E$ has entangled components $E_1,\ldots, E_n,$ there is a sequence of moves after which there are no crossings between $E_j$ and $E_k$ for any $j \neq k.$ By the above discussion, these moves can be performed by ambient isotopies on $E\cup \gamma.$ This yields a link diagram for $E\cup \gamma$ in which all crossings are either within entangled components or involving $\gamma.$ Then, since $\gamma$ still only has crossings of a single type with each component, $E \cup \gamma$ is not entangled.
\end{proof}

The perimeter law bound follows quickly from the previous result.

\begin{proof}[Proof of Perimeter Law]
Note that if $E\cup \gamma$ contains an infinite entangled component then so does a graph obtained by adding finitely many edges of $\paren{\Z^3}^*$ to $E.$ As such, since we are assuming that $p > 1- \pi_c^e,$ the event $F_{\gamma}$ almost surely occurs. Then using the FKG inequality, Lemma~\ref{lemma:perimeter} and Corollary~\ref{cor:zaxis}, we have
\begin{align*}
    \mathbb{P}(U_\gamma)
    &\geq \mathbb{P}\paren{D_{\gamma} \cap F_{\gamma}}\\ 
    &\geq \mathbb{P}\paren{\cap_{i \leq N} K^i \centernot{\xleftrightarrow[]{e}} G}\\
    &\geq \mathbb{P}\paren{K \centernot{\xleftrightarrow[]{e}} G}^{\mathrm{Per}(\gamma)}\\
    &\geq \exp{(-\beta'\mathrm{Per}(\gamma))}.
\end{align*}
\end{proof}

\section{Area Law}\label{sec:area}

The argument to prove an area law for $U_{\gamma}$ in the regime $p < 1- p_c^{\lim}$ is similar to the corresponding proof for the event that $\gamma$ is nullhomologous in the regime $p < 1-p_c\paren{\Z^3}$ so we only give a sketch and note the differences. By Proposition~\ref{prop:vgammadual}, it is enough to consider the event that $\gamma$ is entangled with a set of dual bonds. In~\cite{ACCFR83}, a renormalized construction in slabs is used to produce a single loop that is linked with $\gamma$ when $p < 1-p_c.$ Assuming continuity in slabs for $m$-entangled percolation, a similar renormalization can be performed for $m$-entangled percolation for any fixed $m$ and $p < 1-p_c^m.$ In this case, bonds in the renormalized lattice ensure underlying $m$-entangled paths in the original lattice. Since projections of entangled paths onto $\Z^2$ are connected, entangled paths in slabs overlap in the same way that ordinary paths do.  The only adaptation required in the construction is verifying that components that witness box crossings are entangled, which may require revealing the states of edges outside of the box. However, since $m$ is fixed, the number of additional states that must be revealed is uniformly bounded above, so the dependence remains local. We summarize the properties of this construction in the following theorem.

\begin{theorem}\label{theorem:arearenorm}
    Let $p < 1- p_c^{\lim}.$ For any $m \in \N$ and any $\epsilon > 0,$ there is a renormalized dual bond percolation construction $\mathbb{B}$ consisting of boxes that represent sites and boxes representing bonds between them so that
    \begin{itemize}
        \item The states of bonds of $\mathbb{B}$ are $1$-dependent.
        \item The probability that a bond of $\mathbb{B}$ is open is at least $1-\epsilon.$
        \item If a bond $E$ of $\mathbb{B}$ with endpoints $V$ and $W$ is open, then there is an $m$-entangled subset of $B\cap \paren{E \cup V \cup W}$ that intersects both $V$ and $W.$
        \item For any site $V,$ there is at most one entangled component of $B$ that intersects $V$ and any other sites of $\mathbb{B}.$
    \end{itemize}
\end{theorem}

By considering disjoint renormalized constructions in slabs orthogonal to $\gamma$ as in~\cite{ACCFR83}, it follows from Theorem~\ref{theorem:arearenorm} that there is an area law upper bound for the probability that there is an open renormalized loop that is linked with $\gamma$ when $p < 1- p_c^{\lim}.$ We now verify that such a loop prevents the event $U_{\gamma}.$ We begin by proving a technical lemma that implies that an entangled component contained in the interior of both a sphere and a box is contained in a single component of their intersection.

\begin{lemma}\label{lemma:boxcomponents}
    Let $R \subset \R^3$ be a box and let $\mathcal{S}$ be a piecewise--linear embedded sphere. Suppose $E \subset I\paren{\partial R}$ is a set of bonds. Then there is a collection of embedded spheres $S_1,\ldots,S_k\subset I\paren{\partial R}\setminus E$ so that if $v,w\in E$ then $v$ and $w$ are in the same path component of $R\setminus \mathcal{S}$ if and only if they are in the same path component of $I\paren{\partial R}\setminus \cup_{i=1}^k S_k.$  
\end{lemma}

\begin{proof}
    We describe an iterative process to produce the spheres. We may assume that $\mathcal{S}$ intersects $\partial R$ transversely so $\mathcal{S} \cap \partial R$ is a finite union of disjoint loops. $\mathcal{S}\setminus \mathcal{S}\cap\partial R$ is a finite union of surfaces with boundary at least two of which must be disks. Denote one of those disks by $D_{\alpha}$ and its boundary by $\alpha.$
    $\mathcal{S}$ intersects $\partial R$ transversely and $E$ is a closed subset of the interiors of those sets so we may find an embedding $f:\mathcal{D}^2\times \brac{0,1}\to \R^3$ (where $\mathcal{D}^2$ is a two-dimensional ball) so that $\im f\cap \mathcal{S}=f\paren{S^1\times \brac{0,1}}$, $\im f\cap \partial R=f\paren{\mathcal{D}^2\times \set{1/2}}=D_\alpha,$ and $\im f\cap E=\varnothing.$ Let $D_1=f\paren{\mathcal{D}^2\times \set{0}}$ and $D_2=f\paren{\mathcal{D}^2\times \set{1}}$ and set
    \[\mathcal{S}' \coloneqq \paren{\mathcal{S} \setminus f\paren{S^1\times\paren{0,1}}} \cup D_1 \cup D_2\]
    so $\mathcal{S}'$ is the union of two disjoint spheres so that $\mathcal{S}' \cap \partial R$ has one fewer component than $\mathcal{S} \cap \partial R.$ Note that $v,w\in E$ are in the same path component of $R\setminus \mathcal{S}$ if and only if they are in the same path component of $R\setminus \mathcal{S}'.$ Apply this process to both resulting spheres, continuing until no intersections with $\partial R$ remain. If the resulting collection contains only a single sphere outside of $R,$ we may push it inside of $R.$ Otherwise, we may conclude by discarding any spheres in the exterior of $R.$ 
\end{proof}

\begin{corollary}\label{cor:boxcomponents}
    Let $R \subset \R^3$ be a box and let $\mathcal{S}$ be a piecewise--linear embedded sphere. Suppose $E \subset I\paren{\partial R} \cap I\paren{\mathcal{S}}$ is an entangled set of bonds. Then $E$ is contained in a single component of $R \cap I\paren{\mathcal{S}}.$
\end{corollary}

\begin{prop}\label{prop:entangledrenorm}
    Suppose there is a renormalized loop in $\mathbb{B}$ that is linked with $\gamma.$ Then $U_{\gamma}$ does not occur.
\end{prop}

\begin{proof}
Denote the renormalized sites in the loop by $V_1,\ldots,V_n,V_{n+1}=V_1$ in cyclic order. Let $R_j$ be the union of the boxes corresponding to the renormalized sites $V_j$ and $V_{j+1}$ and the renormalized edge between them. We may choose vertices $w_1,\ldots, w_n,w_{n+1}=w_1$ of $\paren{\Z^3}^*$ so that $w_j$ is entangled to $w_{j+1}$ inside of $R_j.$ Let $E_j$ be an entangled graph containing $w_j$ and $w_{j+1}$ inside of $R_j.$ Set $E=\cup_{j=1}^n E_j$, and suppose that $\mathcal{S}$ is a sphere containing $E$ that does not intersect $\gamma.$ By Corollary~\ref{cor:boxcomponents} $E_j$ is contained in a single component of $\mathcal{S}\cap B_j.$ Thus, there is a connected path of bonds $\xi_j$ connecting $w_j$ to $w_{j+1}$ inside that component of $\mathcal{S}\cap B_j.$ Then the simple loop formed by the segments $\xi_1,\ldots,\xi_j$ is linked with $\gamma$ and contained in the interior of $\mathcal{S}.$ Thus $\gamma$ is also contained in the interior of $\mathcal{S}.$ Since this is true for all such spheres $\mathcal{S},$ $E$ is entangled with $\gamma$ and $U_{\gamma}$ does not occur by Proposition~\ref{prop:vgammadual}.
\end{proof}

The existence of a nontrivial asymptotic area law constant $\alpha\paren{p}$ in the area law regime also follows from the methods of~\cite{ACCFR83}; the same proof of their Proposition 2.4 using tilings of $\gamma$ by smaller loops works in our setting without modification. Combining this observation with Theorem~\ref{theorem:arearenorm} and Proposition~\ref{prop:entangledrenorm} yields the following corollary:

\begin{corollary}\label{cor:arealaw}
    Let $p > 1-p_c^{\mathrm{lim}}.$ Then there is an $\alpha\paren{p} \in \paren{0,\infty}$ such that 
    \[\log\mathbb{P}_p\paren{U_\gamma} \sim -\alpha\paren{p}\mathrm{Area}\paren{\gamma}\]
    as the dimensions of the rectangular loop $\gamma$ tend to $\infty.$
\end{corollary}

It then remains to show continuity of $m$-entanglement in slabs, which involves a slightly more careful adaptation of bond percolation methods.

\subsection{Continuity of $m$-entanglement in slabs}\label{sec:slabs}

Let $p^{k,m}_c$ denote the critical probability for $m$-entangled percolation in $\Z^3 \cap \R^2 \times \brac{-k,k}.$ In this section we show that $p^{k,m}_c \xrightarrow[]{k \to \infty} p^m_c.$ Since plaquettes are not involved in our arguments, we will consider bond percolation with probability $p$ on the primal lattice instead of with probability $1-p$ on the dual in order to simplify notation. We adapt the methods of Grimmett and Marstrand in~\cite{GM90}. Recall that their main theorem was the following:

\begin{theorem}[Grimmett, Marstrand]
Let $F$ be an infinite connected subset of $\Z^d,$ satisfying $p_c\paren{F} < 1.$ Then for every $\eta > 0$ there is a $k \in \N$ such that
\[p_c\paren{2kF + \Lambda_k} \leq p_c\paren{\Z^d} + \eta\,.\]
\end{theorem}

With some minor adjustments, we prove an analogue for $m$-entangled percolation. 

\begin{theorem}\label{theorem:mslabcontinuity}
Let $m \in \N$ and let $F$ be an infinite connected subset of $\Z^3,$ satisfying $p^m_c\paren{F} < 1.$ Then for every $\eta > 0$ there is a $k \in \N$ such that
\[p_c^m\paren{2kF + \Lambda_k} \leq p_c^m\paren{\Z^3} + \eta\,.\]
\end{theorem}

The main technique of~\cite{GM90} is also a renormalization argument, but a significant part of the work is done after the renormalized lattice is constructed. Thus, we will only prove $m$-entanglement versions of a few of the lemmas involved in the renormalization process, and leave the rest of the argument as is. In particular, we will prove versions of Lemmas 5 and 6 of~\cite{GM90}.

First, we recall the strategy of~\cite{GM90}. Large cubes of size $n$ represent the sites of a renormalized lattice, and pairs of sites are connected via paths between several smaller cubes of size $r,$ each of which has all edges open. We review some notation. Let 
\[T(n) = \set{\paren{x_1,x_2,x_3} \in \Z^3 : x_1 = n, 0 \leq x_2,x_3 \leq n} \subset \Lambda_n\,,\]
and
\[T(r,n) = \bigcup_{j = 1}^{2r+1} \set{je_1 + T(n)}\,.\]
In words $T(n)$ contains the ``internal'' vertices of a specified quadrant of one of the faces of the cube $\Lambda_n$ and $T(r,n)$ is a union of possible copies of $\Lambda_r$ attached to the outside of $\Lambda_n$ along $T(n).$
Then we call a translation of $\Lambda_r$ with all edges open a seed (sometimes called an $r$-pad, depending on the source), and define
\[K(r,n) = \set{x \in T(n) : x + e_1 \text{ is in a seed in }T(r,n)}\,.\]
We also consider several notions of boundaries of sets. For a vertex set $V,$ define the exterior vertex boundary 
\[\Delta^{\mathrm{ext}} V \coloneqq \set{v \in \Z^3 \setminus V : v \sim w \text{ for some } w \in V}\] 
and the interior vertex boundary
\[\Delta^{\mathrm{int}} V \coloneqq \set{v \in V : v \sim w \text{ for some } w \in \Z^3 \setminus V}\,.\]
We also define the edge boundary
$\Delta^{\mathrm{edge}} \paren{V} \coloneqq \set{e = \paren{v,w} : v \in V, w \in \Z^3 \setminus V}.$

The desired event for a renormalized box involves finding a path from the component revealed so far to $K\paren{r,n},$ that is, the set of points on the boundary of the box that are adjacent to an external seed. In principle, working with an entanglement model could cause an issue with a renormalization based on box crossings since we may have to reveal edges outside of the box in order to verify that there is an entangled path to its boundary. However, in the case of an $m$-entangled path leading to a seed, this issue is avoidable. Any $m$-entangled path to the boundary of a box in the full lattice contains an internal $m$-entangled path to the boundary of the box with wired boundary conditions. If such a wired path leads to the center of a seed and $r > m,$ the addition of the seed will necessarily make this path $m$-entangled in the full lattice without needing to reveal any additional edges. This is the core of the adaptation of the renormalization that we will now spell out in the following two lemmas.

Similar to the notation from before, for $v,w \in \Z^3,$ write $v \xleftrightarrow[]{m} w$ if $v$ and $w$ are connected by an $m$-entangled set of bonds. For a set $A \subset \Z^3$ we also write $v \xleftrightarrow[]{m} w \text{ in } A$ if $v$ and $w$ are connected by an $m$-entangled set of bonds in $B \cap A$ and we write $v \xleftrightarrow[\text{wired}]{m} w \text{ in } A$ if $v$ and $w$ are connected by an $m$-entangled set of bonds in $\paren{B \cup \Delta^{\mathrm{int}} A} \cap A.$ In other words, these are the events that $v$ and $w$ are connected by $m$-entangled sets of bonds in $A$ with free and wired boundary conditions respectively. 

\begin{lemma}\label{lemma:krnbound}
Let $p > p_c^m.$ Then for every $\eta > 0$ there exist $r,n \in \N$ so that 
\[\mathbb{P}_p\paren{\Lambda_r \xleftrightarrow[]{m} K(r,n) \text{ in } \Lambda_n} \geq 1-\eta\,.\]
\end{lemma}

\begin{proof}
Let $r>2m$ be large enough so that $\mathbb{P}_p\paren{\Lambda_r \xleftrightarrow[]{m} \infty} > 1 - \paren{\frac{1}{3}\eta}^{24}$ and let $M$ be determined later. 

We say that $v$ is locally wired entangled to $w$ in $A$ and write $v \xleftrightarrow[\text{l.w.}]{m} w \text{ in } A$ if they are connected by an $m$-entangled set of bonds in $\paren{B \cap A} \cup \paren{\Delta^{\mathrm{int}} A \cap \paren{\Lambda_m\paren{v} \cup \Lambda_m\paren{w}}}.$ That is to say that there are wired boundary conditions in a neighborhood of $v$ and $w,$ and free boundary conditions anywhere else. We make this distinction in order to exclude entangled paths that pass near the boundary of $A$ before reaching a seed. Notice that a wired path from a vertex to $\Delta^{\mathrm{int}} \Lambda_n$ necessarily contains a subpath to $\Delta^{\mathrm{int}} \Lambda_n$ that is locally wired in $\Lambda_n.$

Let 
\[V(n) = \set{v \in T(n) : \Lambda_r \xleftrightarrow[\text{l.w.}]{m} v \text{ in }\Lambda_n}\,.\]
Suppose that we have chosen $n$ such that $2r+1$ divides $n+1.$ Then $T(n)$ can be divided into $2r \times 2r$ vertex--disjoint squares. If $\abs{V(n)} \geq \paren{2r+1}^2 M,$ then $\Lambda_r$ will be locally wired $m$-entangled to a vertex in at least $M$ of these squares. Then if one of the squares does not intersect the boundary of $T(n)$ and contains the center of one of the faces of a seed, since $r>2m,$ $\Lambda_r$ will necessarily be $m$-entangled to that seed in the full lattice. In the case of a boundary square, we may need additional edges on the boundary to ensure that a wired entangled connection is an entangled connection in the full lattice, but the number of edges is bounded above by $3m^{2},$ the number of vertices in a $m$-neighborhood of a corner of the box. We say that a square $S\subset T(n)$ with side length $2r$ is good if $S$ contains the center of one of the faces of a seed and all edges between vertices in $\Delta^{\mathrm{int}} \Lambda_n$ within distance $m$ of $S$ are open. The set of good squares is then $m$-dependent, so we can take $M$ large enough so that for any set $X$ of at least $M$ different squares,
\[\mathbb{P}_p\paren{\text{At least one of the squares in $X$ is good}} \geq 1-\eta/2.\]

Thus, we only need a sufficient lower bound on $\mathbb{P}_p\paren{\abs{V(n)} \geq (2r+1)^2 M}$ to complete the proof.

Let 
\[U(n) = \set{v \in \Delta^{\mathrm{int}} \Lambda_n : \Lambda_r \xleftrightarrow[\text{l.w.}]{m} v \text{ in }\Lambda_n}\,.\]
By construction, the vertices of $\Delta^{\mathrm{int}} \Lambda_n$ are contained in a union of $24$ copies of $T(n)$ rotated about the coordinate axes. These rotations are symmetries of $\Z^3,$ so since the numbers of paths to different regions of $\Delta^{\mathrm{int}} \Lambda_n$ are positively correlated, applying Harris' Lemma gives
\[\mathbb{P}_p\paren{\abs{U(n)} < 24(2r+1)^2 M} \geq \mathbb{P}_p\paren{\abs{V(n)} < (2r+1)^2 M}^{24}\,.\]
We can bound
\begin{align*}
\mathbb{P}_p\paren{\abs{U(n)} < 24(2r+1)^2 M, \Lambda_n \xleftrightarrow[]{m} \infty} 
&\leq \mathbb{P}_p\paren{0 < \abs{U(n)} < 24(2r+1)^2 M} \\
&\leq (1-p)^{-3 \cdot 24(2r+1)^2 M}\mathbb{P}_p\paren{\abs{U(n)} \neq 0, \abs{U(n+1)} = 0}\\
& \xrightarrow[]{n \to \infty} 0\,,
\end{align*}
because turning off all edges leading out of $U(n)$ ensures that $U(n+1)$ is empty (where we are using that $r>2m$). Then this gives
\begin{align*}
    \mathbb{P}_p\paren{\abs{V(n)} < (2r+1)^2 M} &\leq \mathbb{P}_p\paren{\abs{U(n)} < 24(2r+1)^2 M}^{\frac{1}{24}}\\
    &\leq \paren{\mathbb{P}_p\paren{\abs{U(n)} < 24(2r+1)^2 M, \Lambda_n \xleftrightarrow[]{m} \infty} + \mathbb{P}_p\paren{\Lambda_n \centernot{\xleftrightarrow[]{m}} \infty}}^{\frac{1}{24}}\\
    &\leq \paren{\paren{\frac{1}{3}\eta}^{24} + o(1)}^{\frac{1}{24}}\\
    &\leq \eta/2\,
\end{align*}
for sufficiently large $n.$

Then setting $n$ large enough and using Harris' Lemma gives
\begin{align*}
    \mathbb{P}_p\paren{\Lambda_r \xleftrightarrow[]{m} K(r,n) \text{ in } \Lambda_n} &\geq \mathbb{P}_p\paren{\set{\abs{V(n)} \geq (2r+1)^2 M} \cap \set{\text{One of the $M$ squares is good}}} \\
    &\geq \paren{1 - \frac{\eta}{2}}^2 \geq 1-\eta.
\end{align*}

\end{proof}

For the next lemma we will take the perspective that each edge $e$ of the $\Z^3$ lattice is assigned a $\mathrm{Unif}[0,1]$ random variable $X_e$ and that percolation with probability $q$ is the random subgraph of edges $e$ for which $X_e \le q.$ Then for any constant $\rho \in [0,1],$ we say an edge $e$ is $\rho$-open if $X_e \le \rho.$ Note that this definition makes sense regardless of the percolation parameter $p.$ 

\begin{lemma}\label{lemma:G|H}
Let $p > p_c^m.$ For every $\epsilon,\delta > 0$ there exist $r,n \in \N$ with $n > 2r$ such that the following holds: Let $\Lambda_r \subset W \subset \Lambda_n$ such that $\paren{W \cup \Delta^{\mathrm{ext}}(W)} \cap T(n) = \varnothing.$ Let $\beta : \Delta^{\mathrm{edge}} W \cap E(\Lambda_n) \to [0,1-\delta].$ Let 
\begin{align*}
    G = \bigl\{&\text{$W$ is wired $m$-entangled to $K(r,n)$ by a set of bonds for which each bond $e$ is}\\
    &\text{$(\beta(e) + \delta)$-open if $e \in \Delta^{\mathrm{edge}}(W) \cap E\paren{\Lambda_n}$ and $p$-open otherwise.}\bigr\}
\end{align*}
and
\[H = \set{\text{$e$ is $\beta(e)$-closed for each $e \in \Delta^{\mathrm{edge}}(W) \cap E(\Lambda_n)$}}\,.\]
Then 
\[\mathbb{P}\paren{G \mid H} > 1-\epsilon\,.\]
\end{lemma}

\begin{proof}
Let $\epsilon,\delta > 0$ and then fix $W,\beta$ satisfying the above conditions. Let $t \in \N$ be large enough so that
\[(1-\delta^{6m})^t \leq \frac{\epsilon}{2}\,.\]
Then let $\eta > 0$ so that
\[\eta < \frac{\epsilon (1-p^{6m})^t}{2}\,.\]
By Lemma~\ref{lemma:krnbound}, there exist $r,n$ such that
\[\mathbb{P}_p\paren{\Lambda_r \xleftrightarrow[]{m} K(r,n) \text{ in } \Lambda_n} \geq 1-\eta\,.\]
Now we consider wired $m$-entangled connections between $W$ and $K(r,n).$ We have
\[\mathbb{P}_p\paren{\Delta^{\mathrm{int}} W \xleftrightarrow[\mathrm{wired}]{m} K(r,n) \text{ in } \Lambda_n} \geq \mathbb{P}_p\paren{\Lambda_r \xleftrightarrow[\mathrm{wired}]{m} K(r,n) \text{ in } \Lambda_n} \geq 1-\eta\,.\]

Now for any set $K\subset T(n),$ let $U(K)$ be the set of (possibly non-disjoint) minimal witnesses for $\set{W \xleftrightarrow[\mathrm{wired}]{m} K \text{ in } \Lambda_n}.$ Then let 
\[Y(K) = \set{Z \cap \Delta^{\mathrm{edge}} W : Z \in U(K)}\,.\]
Then since any $m$-entangled connection between $W$ and $K$ must pass through $Y(K),$ we have
\begin{align*}
    \mathbb{P}_p\paren{\Delta^{\mathrm{int}} W \centernot{\xleftrightarrow[\mathrm{wired}]{m}} K(r,n) \text{ in } \Lambda_n} &= \mathbb{P}_p\paren{\text{at least one bond of each element of $Y(K)$ is closed}}\\
    &\geq (1-q^{6m})^{t}\mathbb{P}_p\paren{\abs{Y(K)}\leq t}\,.
\end{align*}
Here we are using that a minimal $m$-entangled set intersects $\Delta^{\mathrm{int}} W$ in at most $6m$ vertices.
Rearranging and setting $K = K(r,n)$, we get
\begin{align}
    \mathbb{P}_p\paren{\abs{Y(K(r,n))} > t} &\geq \paren{1-(1-q^{6m})^{-t}
    \mathbb{P}_p\paren{\Delta^{\mathrm{int}} W \centernot{\xleftrightarrow[]{m}} K(r,n) \text{ in } \Lambda_n}}.
\end{align}
Let $Y = Y(K(r,n)).$ Given $A \in Y,$ we say that $A$ is open (or $p$-open, etc.) if each $e \in A$ is open and closed otherwise. Now conditioned on $Y,$ the random variables $X_e$ are independent for $e \in \cup_{A \in Y} A,$ so we can bound
\[\mathbb{P}_p\paren{\text{each $A \in Y$ is $\beta(e)+\delta$-closed, }\abs{Y} > t \mid H} \leq \paren{1-\delta^{6m}}^t\,.\]
Then we have
\begin{align*}
    \mathbb{P}_p\paren{\text{at least one $A \in Y$ is $\beta(e)+\delta$-open} \mid H} &\geq \mathbb{P}_p\paren{\abs{Y} > t \mid H} - \paren{1-\delta^{6m}}^t\\
    &= \mathbb{P}_p\paren{\abs{Y} > t} - \frac{\epsilon}{2}\\
    &\geq (1-\frac{\epsilon}{2}) - \frac{\epsilon}{2}
\end{align*}
as desired.
\end{proof}

With these lemmas established, the rest of the renormalized construction is identical to the one given in~\cite{GM90}, which we summarize for the case $F = \Z^2$ in the following proposition for future use.

\begin{prop}\label{prop:slabrenorm}
    Let $p > p_c^m.$ Then there are $k,N \in \N$ and a tiling of $\Z^2 \times \set{-k,\ldots,k}$ by $N \times N \times 2k$ boxes so that the set of boxes reached by an $m$-entangled path from the box containing the origin stochastically dominates the connected component of that box in a Bernoulli site percolation on $\Z^2$ with some parameter $q > p_c^{\mathrm{site}}\paren{\Z^2}.$
\end{prop}

This is sufficient to complete the proof of Theorem~\ref{theorem:mslabcontinuity}.

\section{Many Handles in the Intermediate Regime}
\label{sec:genus}
In this section we prove Theorem~\ref{thm:handles}. The idea is to show that $\Lambda_N$ can be split up into $cN^2/\log N$ many $\log N \times k\times 2N$ boxes, a constant fraction of which will be crossed the long way by supercritical $m$-entangled percolation, while $\Lambda_N$ will not be crossed by a connected percolation path. Then the intersection of a plaquette crossing $P'$ with any of these boxes (say, $R$) contributes to the homology of $H_1\paren{P'}.$ Towards this end, we resolve the singularities of $P'\cap R$ to produce a surface that cannot be a sphere and thus must have nontrivial homology (Proposition~\ref{prop:plaquettestosurface}). Recalling the notation in~\eqref{eq:Rboundary} and the paragraph that follows,
in the proof of Proposition~\ref{prop:surfacehandle} we describe a surgery operation on a plaquette crossing of $R$ that produces a surface $S_R \subset R \setminus B_R$ so that $H_1\paren{S_R\cup \partial R}\neq 0$ when $D_1\paren{R}$ is entangled in $B_R.$

We first need a modified version of the deformation retraction given in~\cite{duncan2020homological}. Since we are interested in the $1$-homology of the plaquette surface itself, we would like to ignore the extra $1$-cycles contributed by the $1$-skeleton of $\Z^3$  (recall that, by convention, a $2$-dimensional percolation subcomplex $P$ automatically includes all edges). We therefore define $P^{<2>}$ to be the union of the plaquettes contained in $P.$ The correct topological dual object to $P^{<2>}$ contains not only the dual bonds, but additional higher dimensional cells. 

Consider a general cubical subcomplex $Q \subset \Z^d.$ That is, $Q$ is a union of a set of plaquettes of varying dimensions that is closed under inclusion. Then define $Q^*$ to be the dual complex containing each plaquette of $\paren{\Z^d}^*$ whose dual is not contained in $Q.$ Then we have the following duality result, which generalizes Lemma 7 of ~\cite{duncan2020homological} and may be of independent interest. See Figure~\ref{fig:retraction}.

\begin{figure}
	\includegraphics[width=0.6\textwidth]{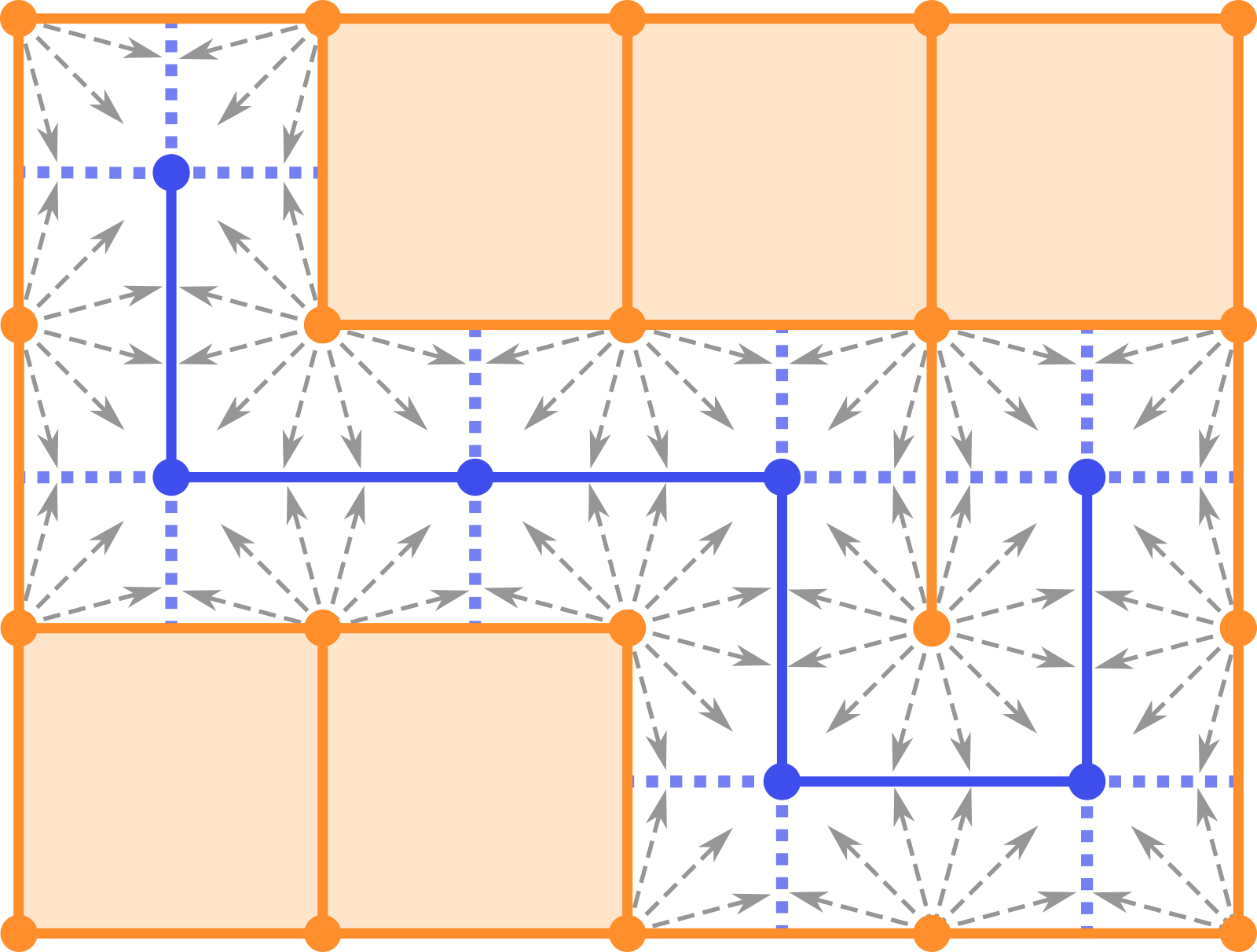}
    \caption{An illustration of the first step of the deformation retraction constructed in Lemma~\ref{lemma:generalretract}. $Q$ is shown in solid blue and $Q^*$ in  orange. The retraction proceeds along the dashed gray arrows from the vertices of $Q^*$ onto the union of $Q$ and a collection of additional half-edges of $\Z^2$ shown by the dotted blue lines. In the next step, the dotted blue half-edges are retracted away from the centers of the orange edges.}  
\label{fig:retraction}
\end{figure}

\begin{lemma}\label{lemma:generalretract}
    For any cubical complex $Q \subset \Z^d,$ $\R^d \setminus Q^*$ deformation retracts onto~$Q.$
\end{lemma}

\begin{proof}
    We perform a series of deformation retractions within $i$-cells of $\Z^d,$ starting from $i=d$ and moving downward in dimension. Consider a $d$-cell $\sigma^d \subset \Z^d.$ Either $\sigma^d \subset Q$ or the dual vertex at the center of $\sigma^d$ is in $Q^*.$ In the former case, we leave $\sigma^d$ fixed, and in the latter case we retract $\paren{\R^d \setminus Q^*} \cap \sigma^d$ to the boundary of $\sigma^d$ via a straight-line radial homotopy from the center point (for an explicit description, see~\cite{duncan2020homological}). The deformation retraction for an $i$-cell $\sigma^i \subset \Z^d$ is similar. If $\sigma^i \subset Q$ we leave it fixed, otherwise the center point of $\sigma^i$ is at the center of the dual cell which is in $Q^*,$ and we can again retract $\paren{\R^d \setminus Q^*} \cap \sigma^i$ to the boundary of $\sigma^i.$ Notice that this is continuous because $Q^*$ is closed under inclusion and so when we retract an $i$-cell it is not contained in a higher dimensional cell. In the case of a $0$-cell $\sigma^0,$ if the dual $d$-cell is in $Q^*$ then again the fact that $Q^*$ is closed under inclusion means that no points of $\R^d \setminus Q^*$ were previously retracted to $\sigma^0,$ so no additional step needs to be taken. It is not difficult to check that this homotopy remains in  $\R^d \setminus Q^*,$ so since we fix $Q$ and remove the interior of each of its missing $i$-cells for $1 \leq i \leq d,$ this gives a deformation retraction onto $Q.$ 
\end{proof}

We now want to turn a plaquette crossing $P'$ of $R$ into a nearby surface in $R \setminus B$ so that we can make use of the classification of surfaces later on. A natural way to avoid singularities caused by plaquettes meeting at a vertex or multiple plaquettes meeting at an edge is to consider an offset of $P',$ or the set of points at a fixed distance from $P'.$ 

Let $F$ be the component of $R \setminus P'$ that contains $D_1^-$ and define 
\[S_{P'}\coloneqq \set{x : d\paren{x,P'} = 1/4} \cap F\,.\]
    
\begin{prop}\label{prop:plaquettestosurface}
    For any plaquette crossing $P'$ of a box $R,$ $S_{P'}$ is a connected surface that separates $D_1.$ Furthermore, if $W$ is the union of $S_{P'}$ and the component of $R \setminus S_{P'}$ that contains $D_1^+$ and $i_* : H_1\paren{S_{P'}} \to H_1\paren{W\cup \partial R}$ and $j_*:H_1\paren{P'}\to H_1\paren{P'\cup \partial R}$ are the maps on homology induced by inclusion then
    \[\rank j_* \geq \rank i_*\,.\]
\end{prop}

\begin{proof}
     A quantity called the weak feature size was introduced in~\cite{chazal2005weak} to quantify how much one can perturb a space without changing its topology. By Proposition 3.6 of~\cite{chazal2005weak}, $P'$ has positive weak feature size since it is piecewise analytic. In fact, one can check that there cannot be singular points of the gradient function defined in within distance $1/4$ of a plaquette subcomplex, so the weak feature size must be at least $1/4.$
     Then by Proposition 3.4 of~\cite{chazal2005weak}, $S_{P'}$ is a surface. It is straightforward to check that $S_{P'}$ contains a connected separating surface, but we show that $S_{P'}$ is connected for completeness.
     
     By Lemma~\ref{lemma:generalretract}, $\R^3 \setminus P'$ deformation retracts onto $\paren{P'}^*.$ Restricting the deformation retraction given in the proof of Lemma~\ref{lemma:generalretract} to $R\setminus P'$ yields a deformation retraction $f:R\setminus P' \times I \to R\setminus P'$ so that $f\paren{R\setminus P',1}=\paren{P'}^*\cap R.$ Since $P'$ is a plaquette crossing $\paren{P'}^*$ has two connected components. Let $T_{1/4} = \set{x\in R : d\paren{x,P'} \leq 1/4}.$ For each point $x\in R\setminus T_{1/4}$ $f\paren{x,t}$ is a path from $x$  to $\paren{P'}^*$ within $R \setminus T_{1/4},$ and so $R \setminus T_{1/4}$ also has two connected components. Now consider part of the Mayer-Vietoris sequence for the decomposition $R = \paren{R \setminus T_{1/4} \cup \partial T_{1/4}} \cup \paren{T_{1/4}}$:
     \[H_1\paren{R} \to H_0\paren{\partial T_{1/4}} \to H_0\paren{\partial T_{1/4} \cup \paren{R \setminus T_{1/4}}} \oplus H_0\paren{T_{1/4}} \to H_0\paren{R} \to 0\,.\]
     Since the sequence is exact, $H_1\paren{R} = 0,$ and $T_{1/4}$ is connected, it follows that $\partial T_{1/4}$ has two connected components. In particular, $S_{P'} = T_{1/4} \cap F$ is connected. 

    Let $V_{P'} \coloneqq \partial \paren{R \setminus \paren{F \cup W}} \cap \partial R$ be the part of $\partial R$ between $S_{P'}$ and $P'.$ We now construct a deformation retraction  $f:W\cup \partial R \to R\setminus F\cup \partial R$ so that $f\mid_{W}$ is a deformation retraction onto $R\setminus F$ and $f\paren{S_{P'}}=P' \cup V_{P'}.$ First, notice that the proof of Lemma~\ref{lemma:generalretract} 
    gives a deformation retraction from $W$ to ${R \setminus F} \cup V_{P'}$ if we stop any point that reaches $\partial R.$ For each cell that intersects $S_{P'},$ the center of that cell is at distance greater than $1/4$ from $P',$ so the radial projection from the center of the cube of $S_{P'}$ onto the boundary is contained in $P' \cup V_{P'}$ or an adjacent lower dimensional cell with the same property that is eventually retracted to $P'\cup V_{P'}.$ Thus, this retraction maps $S_{P'}$ into $P' \cup V_{P'}.$
    
    We extend the map to $W\cup \partial R$ by the identity on $\partial R.$

    Now consider the following commutative diagram, in which the horizontal arrows are the maps induced by inclusions and the vertical arrows are induced by $f$ and its restriction to $S_{P'}$:

    \begin{tikzcd}
        H_1\paren{S_{P'}} \arrow{rr}{} \arrow[swap]{d}{\paren{f \mid_{S_{P'}}}_*} & & H_1\paren{S_{P'} \cup W \cup \partial R} \arrow[swap]{d}{f_*}\\
     H_1\paren{P' \cup V_{P'}} \arrow{r}{}& H_1\paren{P' \cup \partial R} \arrow{r}{}& H_1\paren{\paren{R \setminus F} \cup \partial R}
    \end{tikzcd}

    Since $f$ is a deformation retraction, it induces an isomorphism on homology. Then the rank of the map $H_1\paren{S_{P'}} \to H_1\paren{S_{P'} \cup W \cup \partial R}$ is at most the rank of the map $H_1\paren{P' \cup V_{P'}} \to H_1\paren{P' \cup \partial R}.$ Since $P' \cup V_{P'}$ deformation retracts onto $P'$ by continuing the same construction as above after it reaches $\partial R$, the desired bound follows.
\end{proof}

When we break $\Lambda_N$ into smaller boxes, we will need the following two lemmas describing the intersection of $S_{P'}$ and the smaller boxes.

\begin{lemma}\label{lemma:surfaceintersectionv2}
    Let $P', F$, and $S_{P'}$ be defined with respect to $\Lambda_N$ as in Proposition~\ref{prop:plaquettestosurface}. Then for any box $R = \brac{a,b}\times\brac{c,d}\times \brac{-N,N},$ with $-N \leq a < b \leq N, -N\leq c < d \leq N$ integers, $S_{P'} \cap \partial R$ is a disjoint union of finitely many simple closed curves.
\end{lemma}

\begin{proof}
Since the plaquettes of $P'$ adjacent to $\partial R$ are orthogonal to $\partial R,$ an offset of $P'$ intersects $\partial R$ transversely and therefore $S_{P'} \cap \partial R$ is a disjoint union of curves. It is easily checked that the paths reaching the edges of the box meet up to form simple closed curves, as each point in both $P'$ and an edge of $R$ must be contained in a single plaquette of $P'$ that is orthogonal to both faces meeting at the edge (and no plaquette of $P'$ is adjacent to any of  the eight corner vertices of $R$).
\end{proof}

\begin{lemma}\label{lemma:subsurface}
    Let $\Lambda_N = \cup_{i=1}^M R_i,$ where each $R_i$ meets the hypotheses of the preceding lemma, and the $R_i$'s have disjoint interiors. If $P'$ separates  $D_1\paren{\Lambda_n}$ then for each $i,$ $S_{P'} \cap R_i$ is a disjoint union of finitely many surfaces, at least one of which separates the top of $R_i$ from the bottom. 
\end{lemma}

\begin{proof}
    If $S_{P'}  \cap R_i$ does not contain a separating surface, then the two components of $D_1\paren{\Lambda_N}$ are connected by a path in $R_i \setminus S_{P'},$ so $S_{P'}$ cannot separate the two components of $D_1\paren{\Lambda_N}.$

    By Lemma~\ref{lemma:surfaceintersectionv2}, $S_{P'} \cap \partial R_i$ is a finite union of simple closed curves. Since $S$ is connected, the number of components of $S_{P'}  \cap R_i$ is at most the number of such curves.
\end{proof}

We are now ready to prove the main topological ingredient required for Theorem~\ref{thm:handles}. 

\begin{prop}\label{prop:surfacehandle}
Let $P'$ be a plaquette crossing of $\Lambda_N,$ let $W$ be as in Proposition~\ref{prop:plaquettestosurface} (with $\Lambda_N$ playing the role of $R$ in that statement), and let $R$ be as in Lemma~\ref{lemma:surfaceintersectionv2}. Suppose $B$ contains an entangled crossing of $R.$ Then the map on homology $H_1\paren{S_{P'}\cap R}\to H_1\paren{\paren{W\cap R} \cup\partial R}$ induced by inclusion is nontrivial.

\end{prop}
\begin{proof}
Since $W$ is compact, $\paren{W\cap R}$ breaks up into finitely many disjoint components; let $W_R$ be the component containing $D_1^+\paren{R}.$ For notational simplicity, define $S_R \coloneqq \overline{\partial W_R\setminus \partial R}$ and let $W_R^c \coloneqq R \setminus \paren{\partial R \cup S_R \cup W_R}.$ Then $S_R$ is a surface crossing of $R.$ Since $W_R$ and $S_R\subset S_{P'}\cap R$ are disconnected from the rest of $W\cap R$ and $S_{P'}\cap R,$ respectively, it suffices to show that the induced map on homology $H_1\paren{S_R}\to H_1\paren{W_R\cup \partial R}$ is nontrivial. Our strategy will be to grow $S_R$ and $W_R$ so that the larger sets $S_R'$ and $W_R'$ satisfy $S_R' = \partial W_R'$ and so that the relevant maps on homology induced by inclusion are surjective.

\begin{figure}
	\includegraphics[height=0.4\textwidth]{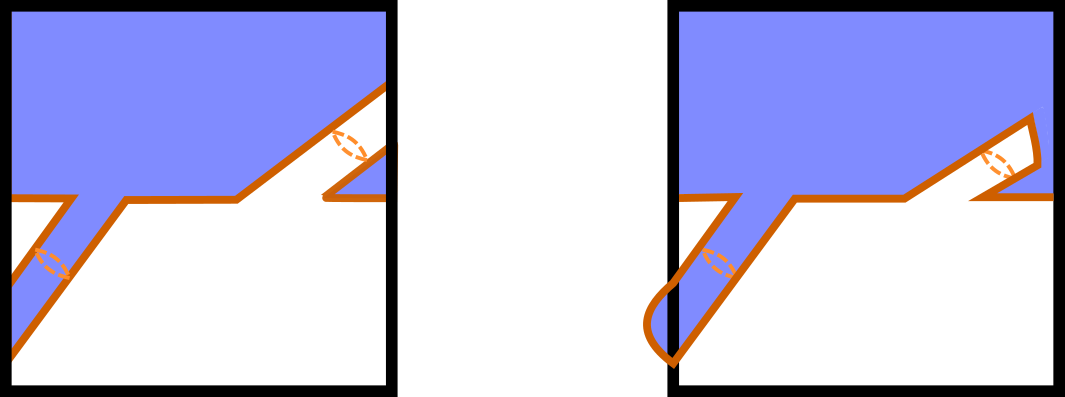}
    \caption{An illustration of the first surgery operation in the proof of Proposition~\ref{prop:surfacehandle}, shown in a cross-section. $S_R$ is depicted in orange and $W_R$ is depicted in blue on the left, whereas $S_R'$ and $W_R'$ are shown in the corresponding colors on the right. The dotted orange lines indicate thin tubes that are bounded by $S$ outside of the cross-section. $S_R$ is obtained from a plane separating $D_1$ by adding two thin tubes, each connecting one half of the box to the boundary of the box in the other half. $S_R$ is modified to yield $S_R'$ by pushing the end of one tube slightly outside of the box and pushing the other slightly inside.}
\label{fig:surgerytype1}
\end{figure}

We now describe two operations that simplify $\partial S_R=S_R\cap \partial R,$ which are illustrated in Figures~\ref{fig:surgerytype1} and~\ref{fig:surgerytype2}. $\partial S_R$ is transverse to $D_2=D_2\paren{R}$ and does not intersect $D_1=D_1\paren{R}.$ As such $\partial S_R$ is a union of disjoint closed curves contained in $D_2.$  We first remove components of $\partial S_R$ that represent trivial elements of $H_1\paren{D_2}.$ If at least one component of $\partial S_R$ represents a trivial element of $H_1\paren{D_2},$ then it contains at least one curve $\gamma$ that bounds a disk $\mathcal{D}$ of $D_2\setminus S_R.$ In this case, set $S_R'' \coloneqq S_R\cup \mathcal{D}.$ If $\mathcal{D}$ separates $W_R$ from the exterior of $R,$ we obtain  

 obtain $S_R'$ from $S_R''$ by pushing $\mathcal{D}$ a small distance outside of $R$ so that $S_R'$ remains a surface transverse to $\partial R.$ Otherwise, we push $\mathcal{D}$ a small distance inside of $R$ to find a surface $S_R'$ that does not intersect $\partial R$ in a neighborhood of $\mathcal{D}.$ 

Then set $W_R'$ to be the closure of the union of $W_R$ and all bounded connected components of $\R^3 \setminus \paren{\partial R \cup \paren{S_R' \setminus R}}$ whose closures intersect $W_R.$ 

$S_R'$ is a space obtained by attaching a $2$-cell to $S_R.$ 
$W_R'$ is either homeomorphic to $W_R,$ or $W_R'$ is homotopy equivalent to $W_R$ with some number of disks attached. In either case, the maps $H_1\paren{S_R} \to H_1\paren{S_R'}$ and $H_1\paren{W_R} \to H_1\paren{W_R'}$ induced by inclusion are surjective.

After iteratively applying this procedure to $S_R,$ eventually all remaining boundary components of the resulting surface $S_R'$ are nontrivial in $H_1\paren{D_2}$; they each must generate $H_1\paren{D_2}$ because they are simple loops. We now describe the operation performed if all of the loops of $S_R$ are of this form. Observe that in this case $\partial R\setminus S_R$ is the disjoint union of two disks --- one containing $D_1^+$ and the other containing $D_1^-$ --- and a (possibly empty) finite collection of circular strips encircling $D_2.$ The spherical cap at the top is contained in $\partial W$ and the adjacent strip is not contained in $W.$ This pattern alternates until one reaches the bottom spherical cap, which is not contained in $\partial W.$ As such, $\partial R\cap S_R$ must contain an odd number of components. Also, $S_R$ is obtained from $\partial W$ by removing a finite collection of circular strips. In the example illustrated in Figure~\ref{fig:tworooms}, $R\setminus S_R$ is the union of two spherical caps and two circular strips.

In the case that there are at least $3$ loops in $\partial R\cap S_R$, we attach disks to two of them at a time. Let $X \subset \partial W_R$ be a strip with $\partial X = \gamma_1 \sqcup \gamma_2 \subset S_R.$ Then we add two disjoint disks $\mathcal{D}_1,\mathcal{D}_2$ below the outside of $R$ attached along $\gamma_1$ and $\gamma_2$ respectively. Now define $S_R' \coloneqq S_R \cup \mathcal{D}_1 \cup \mathcal{D}_2$ and define $W_R'$ to be the union of $W_R$ and the closure of the bounded component of $\R^3 \setminus \paren{X \cup \mathcal{D}_1 \cup \mathcal{D}_2}.$ This has the effect of removing $X$ from $\partial W_R$ and replacing it with $\mathcal{D}_1 \cup \mathcal{D}_2.$ See Figure~\ref{fig:surgerytype2}. If there is only one loop, then $\partial W_R \setminus S_R$ is homeomorphic to a disk, so we set $W_R' \coloneqq W_R$ and $S_R' \coloneqq \partial W_R.$

\begin{figure}
	\includegraphics[height=0.4\textwidth]{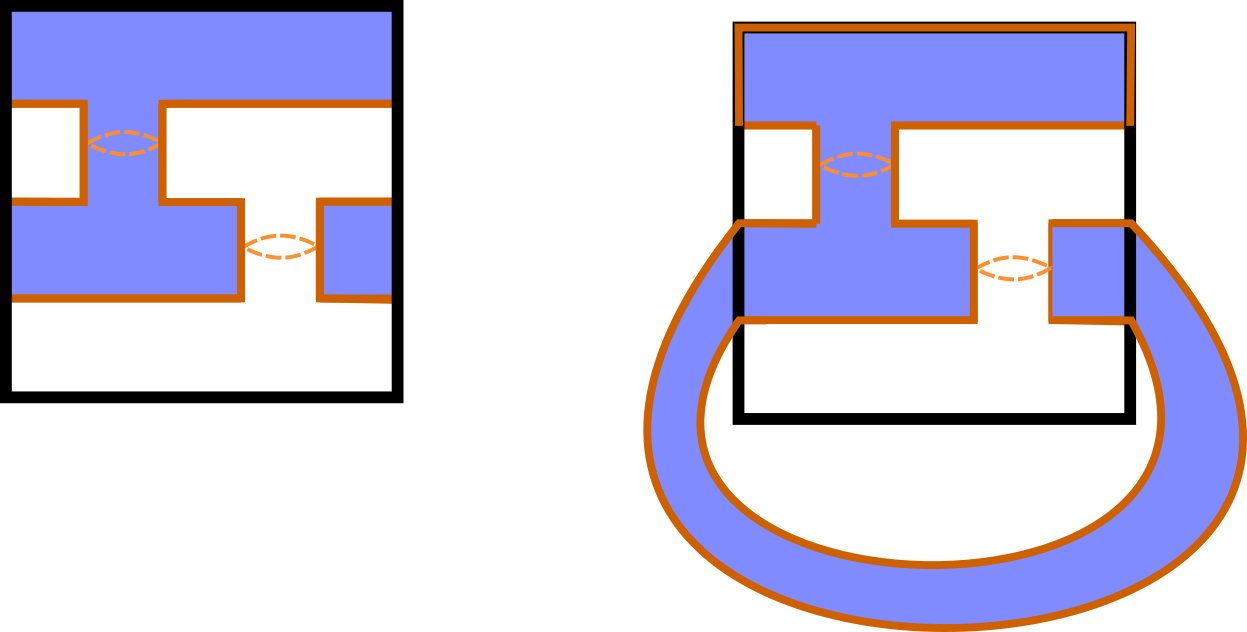}
    \caption{An illustration of the effect of the surgery on the example from Figure~\ref{fig:tworooms}, shown in a cross-section. $S_R$ is shown in orange and $W_R$ is shown in blue on the left, whereas $S_R'$ and $W_R'$ are shown in the corresponding colors on the right. The dotted orange lines indicate thin tubes that are bounded by $S$ outside of the cross-section. $W_R$ is a solid torus and $S_R$ is a cylinder with a disk removed. The homology $H_1\paren{S_R}$ is nontrivial, but the induced map on homology $H_1\paren{S_R}\to H_1\paren{S_R\cup\partial R}$ is not. 
     $S_R'=\partial W_R'$ is a two-sphere obtained from $S_R$ by attaching three disks: two outside the box and the cap on top of the box. $W_R'$ is a three-ball obtained from $W_R$ by attaching a thickened disk around a circular strip.}
\label{fig:surgerytype2}
\end{figure}

We now consider the effect of these operations on the topology of the complex. We claim that for any of the above operations, if $H_1\paren{S_R'} \to H_1\paren{W_R' \cup S_R' \cup \partial R}$ is nontrivial, so is $H_1\paren{S_R} \to H_1\paren{W_R \cup S_R \cup \partial R}.$ To see this, notice that $S_R'$ and $W_R'$ are up to homotopy equivalence always constructed by attaching some number (possibly zero) of disks to $S_R$ and $W_R$ respectively. We therefore have the following commutative diagram
\[\begin{tikzcd}
	{H_1\paren{S_R}} & {H_1\paren{W_R}} & {H_1\paren{W_R \cup \partial R}} \\
	{H_1\paren{S_R'}} & {H_1\paren{W_R'}} & {H_1\paren{W_R' \cup \partial R}}
	\arrow[from=1-1, to=1-2]
	\arrow[from=1-1, to=2-1]
	\arrow[from=1-2, to=1-3]
	\arrow[from=1-2, to=2-2]
	\arrow[from=1-3, to=2-3]
	\arrow[from=2-1, to=2-2]
	\arrow[from=2-2, to=2-3]
\end{tikzcd}\,,\]
in which all vertical maps are surjective. 

Abusing notation slightly, let $S_R'$ and $W_R'$ be the final sets obtained after applying these operations in sequence to the original sets $S_R$ and $W_R.$ By construction $\partial W_R' = S_R'.$ All that remains is to show that $H_1\paren{S_R'} \to H_1\paren{W_R' \cup S_R' \cup \partial R}$ is nontrivial. Since $S_R$ separates $D_1$ in $R \setminus B_R,$ one can check that the operations described above result in a surface $S_R'$ that  separates $D_1$ in $\R^3 \setminus B_R.$

Then since $B_R$ contains an entangled crossing, $S_R'$ cannot be a sphere, so by the classification of surfaces it must be homeomorphic to a connected sum of tori. It follows that $W_R'$ is homeomorphic to a connected sum of solid tori and that $H_1\paren{S_R'} \to H_1\paren{W_R'}$ is surjective and in particular nontrivial. Finally, $\partial R \setminus W_R'$ consists of the union of the bottom of $\partial R,$ which is homeomorphic to a disk, and some number of strips in $D_2 \cap \partial W_R$ between loops of $\partial S_R.$ Since we have attached a disk along each loop of $\partial W_R\cap R,$ each of these loops is nullhomologous in $H_1\paren{W_R'}$ and so the map $H_1\paren{W_R'} \to H_1\paren{W_R' \cup \partial R}$ is injective, completing the proof.
\end{proof}

With the topology worked out, we now need to give a lower bound on the probability of an entangled crossing of a thin box in the long direction. Let $R_{N,a,k} = \brac{0,a\log N} \times \brac{0,k} \times \brac{-N,N}$ and let $H_{N,a,k}$ be the event that there is a $m$-entangled crossing of $R_{N,a,k}.$ 

\begin{lemma}\label{lemma:thincrossing}
    Fix $m \in \N.$  Then for any $p > p_c^{\lim}$ there is a $k = k\paren{m} \in \N$ and a $c\paren{p,m} > 0$ so that
    \[\inf_n \mathbb{P}_p\paren{H_{N,c,k}} > 0\,.\]
\end{lemma}

\begin{proof}
    The renormalization scheme described in Proposition~\ref{prop:slabrenorm} explores the connected component of a single vertex, and in the same way one can explore from a starting set with multiple vertices. In particular, using such an exploration from the top of the box, the present lemma can be reduced to showing that for any $q > p_c^{\mathrm{site}}\paren{\Z^2},$ there is a $b > 0$ so that 
    \[\inf_N \mu_q\paren{\mathcal{R} \cap \set{x = 0} \xleftrightarrow[]{\mathcal{R}}{} \mathcal{R} \cap \set{x = 2N}} > 0\,,\]
    where $\mu_q$ is the measure for site percolation in $\Z^2$ and 
    $\mathcal{R} = \set{0,1,\ldots,2N}\times\set{0,\ldots,\lceil b\log N \rceil}.$
    
    In order to show this, we consider the dual site percolation on the graph $\paren{\Z^2}^\times,$ which has the same vertex set but edges between vertices at $L^\infty$ distance $1$ instead of just nearest neighbors (that is, we add diagonal edges in each square of the usual $\Z^2$ lattice). We couple the primal site percolation subgraph, which we call $X,$ with a dual site percolation $Y$ with parameter $1-q.$ Let $F_l$ be the event that there is an open dual path from $\set{0,\ldots,2N}\times \set{0}$ to $\set{0,\ldots,2N} \times \set{l}$ in $\set{1,\ldots,2N}\times \Z.$ If $\sup_N \mathbb{P}\paren{F_{b\log N}} < 1,$ by a standard duality argument we have the desired bound on the probability of a primal box crossing.

    The probability of $F_l$ is bounded above by the probability there is at least one dual vertex in $\set{1,\ldots,2N}\times \set{0}$ that is in a component of size $l.$ Since $Y$ is subcritical, there is a constant $d>0$ so that 
    \[\mu\paren{0 \xleftrightarrow[Y]{}\partial \Lambda_l} \leq \exp\paren{-d l}\,.\]
    Then by a union bound, we have 
    \[\mu\paren{F_l} \leq N\exp\paren{-d l}\,,\]
    so taking $b = \frac{2}{d}$ so that $l = \lceil\frac{2}{d}\log N\rceil$ gives the desired bound.
\end{proof}

We are now ready to prove Theorem~\ref{thm:handles}.

\begin{proof}[Proof of Theorem~\ref{thm:handles}]    
    Suppose $1-p_c\paren{\Z^3} < p < 1-p_c^{\lim}.$ In this regime, $P$ separates $D_1^+$ from $D_1^-$ with high probability so we will assume that it does.  

    We first show the upper bound. Let $F$ be the union of the components of $\Lambda_N \setminus P$ that intersect $D_1^-$ and let $P'$ be the subset of the plaquette boundary of $F$ containing all plaquettes adjacent to cubes of the connected component containing $D_1^+$ in $\Lambda_n \setminus F.$ $P'$ is a plaquette crossing of $R.$

    Denote by $E\paren{P'}$ the set of edges adjacent to plaquettes of $P'.$ Then
    $$\rank H_1\paren{P'\cup \partial R}\leq \rank H_1\paren{P'} \leq \rank C_1\paren{P'}\leq \abs{E\paren{P'}}\leq 4\abs{P'}\leq 20\abs{F}$$
    where $\abs{F}$ denotes the number of dual sites contained in $F.$ In the last two inequalities we are using the fact that each plaquette contains four edges and each cube has 6 faces, and that a minimal surface cannot contain all of the plaquettes in the boundary of a single cube. We have that
    \[\mathbb{P}_p\paren{\abs{F} \geq m} \leq \mathbb{P}_p\paren{\sum_{i=1}^{N^2} X_i \geq m}\,,\]
    where the $X_i$'s are i.i.d. copies of the random variable distributed as the size of the dual component of an arbitrary vertex in $B.$ Then since $B$ is subcritical, the law of large numbers gives the desired bound.       

    We now turn to the lower bound. Suppose that $P'$ is a plaquette crossing of $\Lambda_N$; such a crossing exists with high probability. We can write $\Lambda_N = \bigcup_{j=1}^{M} R_j \cup Z,$ where $M \coloneqq \lfloor\frac{N}{2a\log N}\rfloor \lfloor \frac{N}{k}\rfloor$ for $a,k$ as in Lemma~\ref{lemma:thincrossing} and each $R_j$ is a translate of $R_{N,c,k}$ by a vector in \[\mathrm{span}_{\Z}\set{\paren{0,c\log N,0},\paren{0,0,k}}\,.\] By Lemma~\ref{lemma:thincrossing}, there is a $\delta > 0$ so that, with high probability, there are entangled top--bottom crossings of at least $M'\coloneqq \ceil{\delta M}$ of the $R_j$'s: enumerate them as $\hat{R}_1,\ldots,\hat{R}_{M'}.$ 

    By Proposition~\ref{prop:plaquettestosurface} it suffices to show that the rank of map on homology $\phi_* : H_1\paren{S_{P'}} \to H_1\paren{W\cup \partial \Lambda_N}$ induced by inclusion is at least $M'.$ We do this by iteratively combining the pieces in smaller boxes using the Mayer-Vietoris sequence.

Let $\Lambda^1,\Lambda^2,\Lambda^3 \subset \Lambda_N$ be boxes so that $\Lambda^1$ and $\Lambda^2$ have disjoint interiors and $\Lambda^3 = \Lambda^1 \cup \Lambda^2.$  Set $W_i = W \cap \Lambda^i$ and $ S_i = S_{P'} \cap \Lambda^i.$

Consider the following commutative diagram, in which the columns are maps induced by inclusions and the rows are from the Mayer--Vietoris sequences for the decompositions $S_3 = S_1 \cup S_2$ and $W_3 \cup \partial \Lambda^1 \cup \partial \Lambda^2 = \paren{W_1 \cup \partial \Lambda^1} \cup \paren{W_2 \cup \partial \Lambda^2}$ respectively:
    
\[\begin{tikzcd}
	{} & {H_1\paren{S_1}\oplus H_1\paren{S_2}} & {H_1\paren{S_3}} \\
	&& {H_1\paren{S_3 \cup \partial \Lambda^3}} \\
	{H_1\paren{\partial \Lambda^1 \cap \partial \Lambda^2}} & {H_1(W_1 \cup \partial \Lambda^1) \oplus H_1\paren{W_2 \cup \partial \Lambda^2}} & {H_1\paren{W_3 \cup \partial \Lambda^1 \cup \partial \Lambda^2}}
	\arrow[from=1-2, to=1-3]
	\arrow["{k_1 \oplus k_2}"', from=1-2, to=3-2]
	\arrow["{i_*}", from=1-3, to=2-3]
	\arrow["{j_*}", from=2-3, to=3-3]
	\arrow["h", from=3-1, to=3-2]
	\arrow["{l_1-l_2}", from=3-2, to=3-3]
\end{tikzcd}\]
    $h$ is trivial because  $H_1\paren{\partial \Lambda^1 \cap \partial \Lambda^2} = 0.$ Thus, by exactness of the bottom row, $l_1 - l_2$ is injective. It follows that 
    \begin{equation}\label{eq:mvboxes}
        \rank i_* \geq \rank \paren{k_1 \oplus k_2} = \rank k_1 + \rank k_2\,.
    \end{equation}   
    For example, if $\Lambda^1$ and $\Lambda^2$ were two of the rectangles $\hat{R}_i$ then we would have $\rank k_1\geq 1$ and   $\rank k_2\geq 1$ by Proposition~\ref{prop:surfacehandle} so it would follow that $\rank i_*\geq 2.$ 

    We now add the $R_j$'s one at a time in each row using (\ref{eq:mvboxes}) and apply  Proposition~\ref{prop:surfacehandle} to each $\hat{R}_i$ and then combine rows adding one at a time using (\ref{eq:mvboxes}). One more similar application of Mayer--Vietoris to $\bigcup_{j=1}^M R_i$ and $Z$ gives the bound
    \[\rank \phi_* \geq M'\,,\]
    completing the proof.
\end{proof}

\bibliographystyle{alpha}
\bibliography{refs}
\end{document}